\documentclass{amsart}
\usepackage[utf8]{inputenc}
\usepackage[english]{babel}
\usepackage{amsmath}
\usepackage{amsfonts}
\usepackage{amssymb}
\usepackage{amsthm}
\usepackage{listings}
\usepackage{hyperref}
\usepackage{mathtools}
\usepackage[all]{xy}
\usepackage{youngtab}
\usepackage{array}
\usepackage{csquotes}
\usepackage{tikz-cd}
\usepackage{stmaryrd}
\usepackage{longtable}
\usepackage{booktabs}
\usepackage{multirow}
\usetikzlibrary{calc}
\usepackage[backrefs,lite]{amsrefs}

\DeclareMathOperator{\rank}{rank}

\DeclareMathOperator{\im}{im}
\DeclareMathOperator{\Cl}{Cl}

\DeclareMathOperator{\codim}{codim}

\DeclareMathOperator{\Spec}{Spec}

\DeclareMathOperator{\lcm}{lcm}

\newcommand{\Pp}{\mathbb{P}}
\newcommand{\C}{\mathbb{C}}

\newcommand{\T}{\mathbb{T}}
\newcommand{\Z}{\mathbb{Z}}
\newcommand{\Q}{\mathbb{Q}}

\newcommand{\K}{\mathbb{K}}
\newcommand{\Oo}{\mathcal{O}}
\newcommand{\Rr}{\mathcal{R}}

\newtheorem{lem}{Lemma}[section]
\newtheorem{proposition}[lem]{Proposition}
\newtheorem{theorem}[lem]{Theorem}

\theoremstyle{definition}
\newtheorem{definition}[lem]{Definition}

\newtheorem{construction}[lem]{Construction}
\newtheorem{example}[lem]{Example}
\newtheorem{reminder}[lem]{Reminder}
\newtheorem{remark}[lem]{Remark}
\newtheorem{corollary}[lem]{Corollary}
\newtheorem{convention}[lem]{Convention}
\newtheorem{classification}[lem]{Classification list}

\title[Gorenstein Fano threefolds with a $\K^*$-action]{gorenstein fano threefolds of picard number one with a $\K^*$-action and maximal orbit quotient $\Pp_2$}

\author{Marco Ghirlanda}
\date{}
\begin{document}

\maketitle
\begin{abstract}
    We classify the non-toric, $\Q$-factorial, Gorenstein,
Fano threefolds of Picard number one with an effective
$\K^*$-action and maximal orbit quotient~$\Pp_2$.
\end{abstract}

\section{Introduction}
\addcontentsline{toc}{chapter}{Introduction}

By a \emph{Fano variety} we mean a normal, projective variety~$X$,
defined over an algebraically closed field~$\K$ of characteristic
zero such that $X$ admits an ample anticanonical divisor.
We contribute to the classification of singular Fano
threefolds $X$ endowed with an effective algebraic torus
action $\T \times X \to X$.
Recall that for $\dim(\T)=3$, that is the \emph{toric case},
there is the complete classification~\cite{MR2760660} 
in the case of at most canonical singularities.
Moreover, for $\dim(\T)=2$ we have complete classifications
in Picard number one for the cases of at most terminal
singularities~\cite{MR3509936} and for the case of
$\Q$-factorial Gorenstein log terminal threefolds ~\cites{MR4509960}.

We focus on $\dim(\T) = 1$, that means $\T = \K^*$. Our
approach uses the \emph{maximal orbit quotient}
from~\cite{MR4031105}.
Consider the invariant open subset $X_0 \subset X$
consisting of all points $x \in X$ with finite isotropy
group $\T_x$.
The inclusion of fields $\K(X)^{\T} \subset \K(X)$ 
gives rise to a rational map
$\pi \colon X_0 \dashrightarrow Y$ 
which is a surjective morphism up to
codimension two w.r.t. $X_0$ and $Y$.
We further focus on the case $Y=\Pp_2$. We fix a divisor
\[
C=C_0 + C_1 + C_2 + C_3 + \dots + C_n,
\]
where $C_0, C_1, C_2 \subset \Pp_2$ are lines in general position and $C_3, \dots, C_n\subset \Pp_2$ are irreducible curves of degrees $d_3,\dots,d_n$, such that all critical values of $\pi$ lie in the support of $C$ and $C$ is minimal with these properties. Furthermore, we have
$$
D_i
=
\pi^*(C_i)
 =
\sum_{j=1}^{n_i} l_{ij} D_{ij},
$$
with $D_{ij} \subset X$ prime divisors and $l_{ij}\in\Z_{>0}$. 
The situation with all $C_i \subset \Pp_2$
being lines has been studied by Hausen, Hische
and Wrobel who provided explicit
classifications~\cites{MR4031105,HW,HW2},
even for higher dimensional analoga.
In the present article, we settle the case
of non-toric, $\Q$-factorial, Gorenstein Fano threefolds of Picard
number one with an effective $\K^*$-action and
$Y=\Pp_2$.
Proposition \ref{we get all} shows that we have to deal
with the following three types:
\[
\begin{array}{lll}
\text{A:} \quad
& 
\pi^*(C_i) = l_{i1}D_{i1}, \ i=0,\ldots,n, 
&
X^{\K^*} \text{ hosts exactly one prime divisor},
\\[8pt]
\text{B:} 
& 
\pi^*(C_0) = l_{01}D_{01} + l_{02}D_{02},
&
\pi^*(C_i) = l_{i1}D_{i1}, i=1,\ldots,n,
\\[2pt]
&
\codim_X(X^\T) \ge 2,    
\\[8pt]
\text{C:} 
& 
d_3\geq 2,\quad\pi^*(C_3) = l_{31}D_{31} + l_{02}D_{32},
&
\pi^*(C_i) = l_{i1}D_{i1}, i=0,1,2,4,\ldots,n,
\\[2pt]
&
\codim_X(X^\T) \ge 2.
\end{array}
\]

Our main result is the following classification theorem.
The families and their varieties are explicitly presented
in terms of Cox ring data in Section \ref{lists section}.
Moreover, we list there the associated anticanonical degrees
and Hilbert functions.

\begin{theorem} \label{big theorem}
We obtain $154$ families of non-toric, $\Q$-factorial, Gorenstein
Fano threefolds of Picard number one with an effective
$\K^*$-action and maximal orbit quotient~$\Pp_2$.
Listed according to the possible $n$ and $d_3,\dots,d_n$,
the numbers of families of Type A, B, C are given by:
{\scriptsize
\begin{center}
\begin{longtable}{c | c | c | c | c }
        $n$ & $d_3,...,d_n$ & \# Type A & \# Type B & \# Type C
        \\
        [5pt]
        \hline
        &&&\\
        $3$ & $1$ & 17 & 56 & 0 \\
        [5pt]
        \hline
        &&&\\
        $3$ & $2$ & 14 & 8 & 22 \\
        [5pt]
        \hline
        &&&\\
        $3$ & $3$ & 7 & 2 & 6 \\
        [5pt]
        \hline
        &&&\\
        $3$ & $4$ & 2 & 0 & 3 \\
        [5pt]
        \hline
        &&&\\
        $3$ & $5$ & 1 & 0 & 0 \\
        [5pt]
        \hline
        &&&\\
        $3$ & $6$ & 0 & 0 & 1 \\
        [5pt]
        \hline
        &&&\\
        $4$ & $1,1$ & 4 & 3 & 0 \\
        [5pt]
        \hline
        &&&\\
        $4$ & $1,2$ or $2,1$ & 1 & 0 & 1 \\
        [5pt]
        \hline
        &&&\\
        $4$ & $2,2$ & 2 & 0 & 1 \\
        [5pt]
        \hline
        &&&\\
        $4$ & $2,3$ or $3,2$ & 1 & 0 & 1 \\
        [5pt]
        \hline
        &&&\\
        $5$ & $1,1,1$ & 1 & 0 & 0
        
    \end{longtable}
    \end{center}
    }
\noindent
Moreover, every non-toric, $\Q$-factorial, Gorenstein Fano threefold of Picard number
one with an effective $\K^*$-action and maximal orbit quotient
$\Pp_2$ belongs to one of these 154 families, and members of different families are not isomorphic as $\K^*$-varieties.
\end{theorem}

\begin{example}[A Fano threefold $X$ in a fake weighted projective space $Z$]
\label{running example 1}
We first construct the ambient space $Z$. Consider the
four-dimensional weighted projective space $W \coloneqq \Pp(1,1,1,1,2)$,
the finite multiplicative group $G \coloneqq \{\pm 1\}^3$ and its 
action on $W$ given in weighted homogeneous coordinates by 
\[
G \times W \ \to \ W,
\qquad  
\zeta\cdot [z] \ = \ [z_0,\, \zeta_1z_1,\, \zeta_2z_2,\, \zeta_3z_3,\, \zeta_1\zeta_2\zeta_3z_4].
\]
Then the quotient $Z := W/G$ is a fake weighted projective space.
In particular, $Z$ inherits from $W$ the structure of a toric
variety, that means that we have an open embedding $\T^4 \subseteq W$ 
such that the group structure of the four-torus $\T^4$ extends to
an action on $W$. Now set
\[
X
\ \coloneqq \
\overline{V_{\T^4}(1+t_2^2+t_3^2+t_4)}
\ \subseteq \
Z,
\]
where $t_2,t_3,t_4$ denote the last three coordinates on $\T^4$.
Observe that the hypersurface $X$ is invariant under scaling the
first coordinate of $\T^4$ and thus comes with a $\K^*$-action.
Note that the inverse image of $X$ under $\gamma \colon W \to W/G = Z$
is given in weighted homogeneous coordinates as 
\[
\gamma^{-1}(X) \ = \ V(z_1^4+z_2^4+z_3^4+z_4^2).
\]  
This allows us to apply the machinery around Construction~\ref{p2cons T1},
showing immediately that the $\K^*$-variety $X$ is a Gorenstein Fano
threefold of Picard number one.
Moreover, we can easily describe the fixed point set of the
$\K^*$-action; in terms of weighted homogeneous coordinates it
is 
\[
X^{\K^*} \ = \ X \cap Z^{\K^*},
\qquad  
Z^{\K^*} \ = \ V(T_0) \cup \{[1,0,0,0,0]\}.
\]
It follows that $X^{\K^*}$ has two connected components, a prime
divisor and an isolated fixed point. The sets
$X_0\coloneqq  X \setminus X^{\K^*}$ and $Z_0\coloneqq Z\setminus Z^{\K^*}$
are unions of non-trivial $\K^*$-orbits and we obtain
a commutative diagram involving geometric quotients
\[
\xymatrix{
&
X_0
\ar@{}[r]|\subseteq 
\ar[d]_{/\K^*}^{\pi}
&
Z_0
\ar[d]_{/\K^*}^{\pi \colon [z_0,\ldots,z_4] \mapsto [z_1^2, \ldots, z_4^2]}
&
\\
\Pp_2
\ar@{}[r]|{\cong \quad}
&
X_0/\K^*
\ar@{}[r]|\subseteq  
&
Z_0/\K^*
\ar@{}[r]|{= \quad}
&
\Pp(1,1,1,2)
}
\]
where $X/\K^*$ sits in $Z/\K^*$ as  $V(T_1^2+T_2^2+T_3^2+T_4)$ 
in $\Pp(1,1,1,2)$ and the identification of the projective
plane $\Pp_2$ with the quotient $X/\K^*$ goes via the isomorphism
\[
\imath \colon \Pp_2 \ \to \ X/\K^*, \quad
[u_0,u_1,u_2] \ \mapsto \ [u_0,u_1,u_2,-u_0^2-u_1^2-u_2^2].
\]  
The critical curves on $\Pp_2 = X/\K^*$ are $C_j \coloneqq \imath^{-1}(V(T_j))$,
where $j = 1,\ldots,4$.
Thus, $X$ is of Type A with $n=3$ and $d_3=2$. Moreover, the general isotropy group $\K^*_x$
over each $C_i$ is cyclic of order $2$.
\end{example}

    Recall that, for a normal variety $X$ with finitely generated class group $\Cl(X)$, the Cox ring is defined as follows:
    $$
    \Rr(X)\coloneqq\bigoplus_{[D]\in\Cl(X)}\Gamma(X,\Oo(D)).
    $$
    When $X$ is $\Q$-factorial and has Picard number one the Cox ring, together with its grading, determines uniquely the variety $X$ (see \cite{MR3307753} for more details). 
    
    \begin{example}\label{running example 2}
    We continue Example \ref{running example 1}.    With $w_i\coloneqq \deg(z_i)$, we have
    $$
    \Rr(X)=\dfrac{\K[z_0,\dots,z_4]}{\langle z_1^4+z_2^4+z_3^4+z_4^2\rangle},\quad Q:=[w_0,\dots,w_4]=\begin{bmatrix}
            1 & 1 & 1 & 1 & 2 \\
            \bar{0} & \bar{1} & \bar{0} & \bar{0} & \bar{1} \\
            \bar{0} & \bar{0} & \bar{1} & \bar{0} & \bar{1} \\
            \bar{0} & \bar{0} & \bar{0} & \bar{1} & \bar{1}
        \end{bmatrix}.
    $$
    By \cite[Theorem 4.4.2.2]{MR3307753} and \cite[Lemma 2.5]{MR3667033}, if $X$ admits an effective $(\K^*)^2$-action there exists a trinomial $f\in\K[z_0,\dots,z_4]$ such that $\Rr(X)$ is equivariantly isomorphic with the graded algebra
    $$
    R\coloneqq\dfrac{\K[z_0,\dots,z_4]}{\langle f(z_0,\dots,z_4)\rangle}.
    $$
    Since $\Rr(X)_{w_i}=\langle z_i\rangle$ for $i=0,\dots,4$, this is impossible. It follows that $X$ does not admit an effective $(\K^*)^2$-action.
    \end{example}

The article is organized as follows. In Section \ref{construction section}, we show how to construct all the $\Q$-factorial, projective threefolds of Picard number one with maximal orbit quotient $\Pp_2$. In Section \ref{classification section}, we determine combinatorial criterions for the Fano and Gorenstein conditions. In Section \ref{proof of theorem}, we prove Theorem \ref{big theorem}. In Section \ref{lists section}, we describe each family in terms of Cox ring data, and provide anticanonical degree and Hilbert function of its varieties.

\section{Projective threefolds with a $\K^*$-action} \label{construction section}
We provide a method to explicitely construct all $\Q$-factorial, projective threefolds of Picard number one with maximal orbit quotient $\Pp_2$, see Construction \ref{p2cons T1} and Proposition \ref{we get all}. We assume basic knowledge of toric geometry, as presented in \cite{toricbook}. We also recall the notion of fake weighted projective space:

\begin{reminder} \label{fwps rem}
A \emph{fake weighted projective space} is a $\Q$-factorial, projective toric variety $Z$ of Picard number one. Every fake weighted projective space arises from a \emph{generator matrix}, that means an $n\times(n+1)$ matrix
$$
P:=[v_0,\dots,v_n],
$$
such that $v_0,\dots,v_n$ are pairwise distinct primitive vectors generating $\Q^n$ as a convex cone. Concretely, there is a unique fan $\Sigma(P)$ having the columns of $P$ as its primitive ray generators, and the associated toric variety $Z(P)$ is a fake weighted projective space.
\end{reminder}

\begin{construction}\label{p2cons T0}
The input data are a degree vector and a family of pairwise distinct, irreducible, homogeneous polynomials:
$$d=(d_3,\dots,d_n)\in\Z_{>0}^{n-2},\hspace{4pt} g=(g_3,\dots,g_n), \hspace{4pt} g_i\in\K[T_0,T_1,T_2],\hspace{4pt} \deg{g_i}=d_i.$$ With this data we associate an embedding of the projective plane
$$
\Pp_2\cong Y(d,g)=V(T_3+g_3,\dots,T_n+g_n)\subset\Pp(1,1,1,d),
$$
Moreover, for the ambient weighted projective space $\Pp(1,1,1,d)$ we provide the generator matrix
$$B = [u_0,\dots,u_n]=
\begin{bmatrix}
-1 & 1 & 0 & 0 & \dots & 0\\
-1 & 0 & 1 & 0 & \dots & 0\\
-d_3 & 0 & 0 & 1 & \dots & 0\\
\vdots & \vdots & \vdots & \vdots & \ddots & \vdots\\
-d_n & 0 & 0 & 0 &\dots & 1\\
\end{bmatrix}.$$
\end{construction}

\begin{example}\label{ex0}
    Let $d=(d_3)=(2)$, $g=(g_3)=(T_0^2+T_1^2+T_2^2)$. Then the associated embedding and generator matrix are
    $$
    \Pp_2\cong V(T_0^2+T_1^2+T_2^2+T_3)\subset\Pp(1,1,1,2), \hspace{20pt} B=
    \begin{bmatrix}
        -1 & 1 & 0 & 0 \\
        -1 & 0 & 1 & 0 \\
        -2 & 0 & 0 & 1\\
    \end{bmatrix}.
    $$
\end{example}

\begin{remark}
        The embedding $Y(d,g)\subset \Pp(1,1,1,d)$ from Construction \ref{p2cons T0} realizes $\Pp_2$ an \emph{explicit variety} in the sense of \cite[Construction 2.8]{MR4031105}. 
\end{remark}

\begin{proposition}\label{p2 explicit}
    Every realization of $\Pp_2$ as an explicit variety arises from Construction \ref{p2cons T0}, up to a homogeneous change of coordinate.
\end{proposition}

\begin{proof}
    Given a realization of $\Pp_2$ as an explicit variety, \cite[Remark 2.11]{MR4031105} and \cite[Theorem 3.2.1.4]{MR3307753} provide us with a commutative diagram
     \[\xymatrix@=35pt{
    \K^3\ar@{^{(}->}[r]_{\bar{\iota}} & \K^{n+1} \\
    \K^3\setminus\{0\} \ar@{^{(}->}[u] \ar[d]^p_{\sslash H} \ar@{^{(}->}[r] & \hat{Z} \ar@{^{(}->}[u] \ar[d]^{p_Z}_{\sslash H}\\
    \Pp_2 \ar@{^{(}->}[r]_{\iota} & Z
    }\]
    where $H$ is the characteristic quasitorus and $p$, $p_z$ are characteristic spaces. In particular,
    $$\Cl(Z)=\Cl(\Pp_2)=\Z.$$ As $\Pp_2$ is complete, we may assume $Z$ to be complete as well. Thus, we may assume $Z=\Pp(d_0,\dots,d_n)$ to be a weighted projective space. Consequently, $\hat{Z}=\K^{n+1}\setminus\{0\}$ and $H=\K^*$. Furthermore, we have
    $$
    \Rr(\Pp_2)=\Rr(\Pp(d_0,\dots,d_n))/I(\bar{\iota}(\K^3)),
    $$
    and the $f_i\coloneqq \bar{\iota}^*(T_i)$ for $i=0,\dots,n$ are pairwise distinct, homogeneous, irreducible polynomials of degree $d_0,\dots,d_n$ respectively, generating $\Rr(\Pp_2)$. In particular, three of the generators must be of degree one and linearly independent. Hence, up to a homogeneous change of coordinate, we can assume $f_i=T_i$ for $i=0,1,2$. It follows that $\bar{\iota}(\K^3)\subset \K^{n+1}$ is a graph, and we have
    $$I(\bar{\iota}(\K^3))=\langle T_3-f_3,\dots,T_n-f_n\rangle.$$
    We conclude that $\iota(\Pp_2)=Y(d,g)\subset\Pp(1,1,1,d)$, where $Y(d,g)$ arises from Construction \ref{p2cons T0} with $d=(d_0,\dots,d_n)$ and $g=(-f_3,\dots,-f_n)$.
\end{proof}

We build up a new matrix $P$ from $B$, distinguishing three types:

\begin{construction}[Type A]\label{cons A}
    Take the embedded plane $Y(d,g)\subset \Pp(1,1,1,d)$ and its associated degree matrix $B=[u_0,\dots,u_n]$ as in Construction \ref{p2cons T0}.
    Fix $l_{01},\dots,l_{n1}\in\Z_{>0}$ and $d_{01},\dots,d_{n1}\in\Z_{\geq 0}$, such that the vectors $v_{i1}\coloneqq[l_{i1}u_i,d_{i1}]$ are primitive and
    $$ \dfrac{d_{01}}{l_{01}}-\dfrac{d_{11}}{l_{11}}-\dfrac{d_{21}}{l_{21}}-\sum_{j=3}^n d_j\dfrac{d_{j1}}{l_{j1}}>0,
    $$
where the $d_j$'s are the entries of the degree vector $d$.
These data give rise to an $(n+1)\times (n+2)$ generator matrix in the sense of Remark \ref{fwps rem}:
\begin{align*}
P=[v_{01},\dots,v_{n1},v_1]&=
\begin{bmatrix}
l_{01}u_0 & l_{11}u_1 & \dots & l_{n1}u_n & 0 \\
-d_{01} & d_{11} & \dots & d_{n1} & 1
\end{bmatrix}.
\end{align*}

Moreover, set $K\coloneqq\Z^{n+2}/\im(P^*)$ and
    denote by $Q\colon\Z^{n+2}\rightarrow K$ the projection and by $e_{01},\dots,e_{n1},e_1\in\Z^{n+2}$ the canonical basis vectors.
    Then we define the $K$-graded algebra
    {\small
    \begin{align*}
    & R(g,P)\coloneqq\K[T_{01},\dots,T_{n1},S_1]/\langle h_3,\dots,h_n\rangle,\hspace{10pt} h_i\coloneqq  T_{i1}^{l_{i1}}+g_i(T_{01}^{l_{01}},T_{11}^{l_{11}},T_{21}^{l_{21}}),\\
    & \deg(T_{i1})\coloneqq \omega_{i1}\coloneqq Q_P(e_{i1}),\\
    & \deg(S_{1})\coloneqq \omega_{1}\coloneqq Q_P(e_{1}).
    \end{align*}
    }
\end{construction}

\begin{example} \label{ex1}
    Set $d=(d_3)=(2)$, $g=(g_3)=(T_0^2+T_1^2+T_2^2)$. As in Example \ref{ex0}, we have
    $$
    B=[u_0,u_1,u_2,u_3]=
    \begin{bmatrix}
        -1 & 1 & 0 & 0 \\
        -1 & 0 & 1 & 0 \\
        -2 & 0 & 0 & 1\\
    \end{bmatrix}.
    $$
    Then $(l_{01},\dots,l_{31})=(2,2,2,2)$, $(d_{01},\dots,d_{31})=(5,1,1,1)$ satisfy the conditions of Construction \ref{cons A}, and we have
    $$
    P=
    \begin{bmatrix}
        2u_0 & 2u_1 & 2u_2 & 2u_3 & 0 \\
        -5 & 1 & 1 & 1 & 1
    \end{bmatrix} =
    \begin{bmatrix}
        -2 & 2 & 0 & 0 & 0 \\
        -2 & 0 & 2 & 0 & 0 \\
        -4 & 0 & 0 & 2 & 0 \\
        -5 & 1 & 1 & 1 & 1
    \end{bmatrix}.
    $$
    Moreover, observe that $K\cong \Z\oplus(\Z/2\Z)^3$. The $K$-graded algebra $R(g,P)$ and its degree matrix are given by
    \begin{align*}
    &R(g,P)=\K[T_{01},T_{11},T_{21},T_{31},S_1]/\langle T_{01}^4+T_{11}^4+T_{21}^4+T_{31}^2\rangle,\\
    & [\deg(T_{01}),\dots,\deg(T_{31}),\deg(S_1)]=\begin{bmatrix}
            1 & 1 & 1 & 1 & 2 \\
            \bar{0} & \bar{1} & \bar{0} & \bar{0} & \bar{1} \\
            \bar{0} & \bar{0} & \bar{1} & \bar{0} & \bar{1} \\
            \bar{0} & \bar{0} & \bar{0} & \bar{1} & \bar{1}
        \end{bmatrix}.
    \end{align*}
\end{example}

\begin{construction}[Type B]\label{cons B}
    Take the embedded plane $Y(d,g)\subset \Pp(1,1,1,d)$ and its associated degree matrix $B=[u_0,\dots,u_n]$ as in Construction \ref{p2cons T0}. Fix $l_{01},\dots,l_{n1},l_{02}\in\Z_{>0}$ and $d_{01},\dots,d_{n1},d_{02}\in\Z_{\geq 0}$, such that the vectors $v_{i1}\coloneqq[l_{i1}u_i,d_{i1}]$ are primitive and
    $$\dfrac{d_{02}}{l_{02}}>\dfrac{d_{11}}{l_{11}}+\dfrac{d_{21}}{l_{21}}+d_3\dfrac{d_{31}}{l_{31}}+\dots+d_n\dfrac{d_{n1}}{l_{n1}}>\dfrac{d_{01}}{l_{01}},$$
    where the $d_j$'s are the entries of the degree vector $d$. These data give rise to an $(n+1)\times (n+2)$ generator matrix in the sense of Remark \ref{fwps rem}:
$$
P=[v_{01},\dots,v_{n1}]=
\begin{bmatrix}
l_{01}u_0 & l_{02}u_0 & l_{11}u_1 & \dots & l_{n1}u_n & 0 \\
-d_{01} & -d_{02} & d_{11} & \dots & d_{n1} & 1
\end{bmatrix}.
$$
Moreover, set $K\coloneqq\Z^{n+2}/\im(P^*)$ and
    denote by $Q\colon\Z^{n+2}\rightarrow K$ the projection and by $e_{01},e_{02},e_{11}\dots,e_{n1}\in\Z^{n+2}$ the canonical basis vectors.
    Then we define the $K$-graded algebra
     {\small
    \begin{align*}
    & R(g,P)\coloneqq\K[T_{01},\dots,T_{n1}]/\langle h_3,\dots,h_n\rangle,\hspace{10pt} h_i\coloneqq  T_{i1}^{l_{i1}}+g_i(T_{01}^{l_{01}}T_{02}^{l_{02}},T_{11}^{l_{11}},T_{21}^{l_{21}}),\\
    & \deg(T_{ij})\coloneqq \omega_{ij}\coloneqq Q_P(e_{ij}).
    \end{align*}
    }
\end{construction}

\begin{example}
    Set $d=(d_3)=(2)$, $g=(g_3)=(T_0^2+T_1^2+T_2^2)$. As in Example \ref{ex0}, we have
    $$
    B=[u_0,u_1,u_2,u_3]=
    \begin{bmatrix}
        -1 & 1 & 0 & 0 \\
        -1 & 0 & 1 & 0 \\
        -2 & 0 & 0 & 1\\
    \end{bmatrix}.
    $$
    Then $(l_{01},l_{02},l_{11},l_{21},l_{31})=(1,2,1,4,2)$, $(d_{01},d_{02},d_{11},d_{21},d_{31})=(1,3,0,1,1)$ satisfy the conditions of Construction \ref{cons B}, and we have
    $$
    P=
    \begin{bmatrix}
        -u_0 & -2u_0 & u_1 & 4u_2 & 2u_3 \\
        -1 & -3 & 0 & 1 & 1
    \end{bmatrix}=
    \begin{bmatrix}
        -1 & -2 & 1 & 0 & 0 \\
        -1 & -2 & 0 & 4 & 0 \\
        -2 & -4 & 0 & 0 & 2 \\
        -1 & -3 & 0 & 1 & 1 
    \end{bmatrix}.
    $$
    Moreover, observe that $K\cong \Z\oplus\Z/2\Z$. The $K$-graded algebra $R(g,P)$ and its degree matrix are given by
    \begin{align*}&R(g,P)=\K[T_{01},\dots,T_{31},S_1]/\langle T_{01}^2T_{02}^4+T_{11}^2+T_{21}^8+T_{31}^2\rangle,\\
    &[\deg(T_{01}),\dots,\deg(T_{31}),\deg(S_1)]=\begin{bmatrix}
    1 & 4 & 2 & 4 & 1 \\
    \bar{1} & \bar{0} & \bar{1} & \bar{0} & \bar{0}
    \end{bmatrix}.
    \end{align*}
\end{example}

\begin{construction}[Type C, only if $d_3>1$]\label{cons C}
Take the embedded plane $Y(d,g)\subset \Pp(1,1,1,d)$ and its associated degree matrix $B=[u_0,\dots,u_n]$ as in Construction \ref{p2cons T0}. Fix $l_{01},\dots,l_{n1},l_{32}\in\Z_{>0}$ and $d_{01},\dots,d_{n1},d_{32}\in\Z_{\geq 0}$, such that the vectors $v_{i1}\coloneqq[l_{i1}u_i,d_{i1}]$ are primitive and
    $$\dfrac{d_3d_{32}}{l_{32}}>\dfrac{d_{01}}{l_{01}}-\dfrac{d_{11}}{l_{11}}-\dfrac{d_{21}}{l_{21}}-d_4\dfrac{d_{41}}{l_{41}}+\dots+d_n\dfrac{d_{n1}}{l_{n1}}>\dfrac{d_3d_{31}}{l_{31}},$$
    where the $d_j$'s are the entries of the degree vector $d$. These data give rise to an $(n+1)\times (n+2)$ generator matrix in the sense of Remark \ref{fwps rem}:
$$
P=[v_{01},\dots,v_{n1}]=
\begin{bmatrix}
l_{01}u_0 & l_{11}u_1 & l_{21}u_2 & l_{31}u_3 & l_{32}u_3 & \dots & l_{n1}u_n & 0 \\
-d_{01} & d_{11} & d_{21} & d_{31} & d_{32} & \dots & d_{n1} & 1
\end{bmatrix}.
$$
with associated fan $\Sigma$ and fake weighted projective space $Z_\Sigma$.
Moreover, set $K\coloneqq\Z^{n+2}/\im(P^*)$ and
    denote by $Q\colon\Z^{n+2}\rightarrow K$ the projection and by $e_{01},\dots,e_{31},$ $e_{32}\dots,e_{n1}\in\Z^{n+2}$ the canonical basis vectors.
    Then we define the $K$-graded algebra
     {\small
    \begin{align*}
    & R(g,P)\coloneqq\K[T_{01},\dots,T_{n1}]/\langle h_3,\dots,h_n\rangle,\hspace{2pt} h_i\coloneqq \begin{cases}
        T_{31}^{l_{31}}T_{32}^{l_{32}}+g_3(T_{01}^{l_{01}},T_{11}^{l_{11}},T_{21}^{l_{21}}) &i=3\\
        T_{i1}^{l_{i1}}+g_i(T_{01}^{l_{01}},T_{11}^{l_{11}},T_{21}^{l_{21}}) &i\geq4
    \end{cases}, \\
    & \deg(T_{ij})\coloneqq \omega_{ij}\coloneqq Q_P(e_{ij}).
    \end{align*}
    }
\end{construction}

\begin{example}
    Set $d=(d_3)=(2)$, $g=(g_3)=(T_0^2+T_1^2+T_2^2)$. As in Example \ref{ex0}, we have
    $$
    B=[u_0,u_1,u_2,u_3]=
    \begin{bmatrix}
        -1 & 1 & 0 & 0 \\
        -1 & 0 & 1 & 0 \\
        -2 & 0 & 0 & 1\\
    \end{bmatrix}.
    $$
    Then  $(l_{01},l_{11},l_{21},l_{31},l_{32})=(4,1,2,1,4)$, $(d_{01},d_{11},d_{21},d_{31},d_{32})=(3,0,1,0,1)$ satisfy the conditions of Construction \ref{cons B}, and we have
    $$
    P=
    \begin{bmatrix}
        -4u_0 & u_1 & 2u_2 & u_3 & 4u_3 \\
        -3 & 0 & 1 & 0 & 1
    \end{bmatrix} =
    \begin{bmatrix}
        -4 & 1 & 0 & 0 & 0 \\
        -4 & 0 & 2 & 0 & 0 \\
        -8 & 0 & 0 & 1 & 4 \\
        -3 & 0 & 1 & 0 & 1 
    \end{bmatrix}.
    $$
     Moreover, observe that $K\cong \Z\oplus\Z/2\Z$. The $K$-graded algebra $R(g,P)$ and its degree matrix are given by
     \begin{align*}&R(g,P)=\K[T_{01},T_{11},T_{21},T_{31},T_{32}]/\langle T_{01}^8+T_{11}^2+T_{21}^4+T_{31}T_{32}^4\rangle,\\
    &Q\coloneqq[\deg(T_{01}),\dots,\deg(T_{31}),\deg(T_{32})]=\begin{bmatrix}
    2 & 1 & 4 & 1 & 4 \\
    \bar{0} & \bar{0} & \bar{0} & \bar{1} & \bar{1}
    \end{bmatrix}.
    \end{align*}
\end{example}

The following is a special case of \cite[Construction 3.5]{MR4031105}, adapted to our needs:

\begin{construction} \label{p2cons T1}
    Let the polynomial family $g$, the associated degree vector $d$ and the integer matrix $P$ be as in Construction \ref{cons A}, \ref{cons B} or \ref{cons C}. We have a commutative diagram
    \[\xymatrix@=35pt{
    X(g,P)\ar@{-->}[d] \ar@{^{(}->}[r] & Z(P)\ar@{-->}[d]\\
    \Pp_2\cong Y(g,P)\ar@{^{(}->}[r] & \Pp(1,1,1,d_3,\dots,d_n)
    }\]
    where the rational map $Z(P)\dashrightarrow \Pp(1,1,1,d_3,\dots,d_n)$ is given by the projection $\T^n\times\K^*\rightarrow \T^n$ of the respective acting tori $\T(P)=\T^n\times \K^*$ and $\T_\Delta=\T^n$, and we define
    $$
    X=X(g,P)\coloneqq \overline{(Y(g,P)\cap\T^n)\times \K^*}\subset Z(P).
    $$
    Then $X\subset Z(P)$ is a projective threefold, invariant under the action of the $1$-dimensional subtorus $\T=\{1_ {\T^n}\}\times\K^*$ of the acting torus $\T(P)=\T^n\times \K^*$ of $Z(P)$.
\end{construction}

\begin{example}
    With the encoding data $g$ and $P$ of Example \ref{ex1}, $X(g,P)$ coincides with the Fano threefold of Example \ref{running example 1}.
\end{example}

From \cite[Proposition 3.7]{MR4031105} we deduce the following:

\begin{proposition} \label{T-conditions prop}
    Let $X\coloneqq X(g,P)$, $R(g,P)$ and $K=\Z^{n+2}/\im(P^*)$ arise from Construction \ref{p2cons T1}.
    The divisor class group and Cox ring of $X$ are
        $$
        \Cl(X)=K,\hspace{10pt}\Rr(X)=R(g,P),
        $$ provided that $R(g,P)$ is $K$-integral and the $T_{ij}$ define pairwise non-associated $K$-primes. 
    In particular, in this case, $X$ is complete intersection, $\Q$-factorial and has Picard number one.
\end{proposition}

\begin{remark}
    A variety $X=X(g,P)$ arising from Construction \ref{p2cons T1} such that $R(g,P)$ is $K$-integral and the $T_{ij}$ define pairwise non-associated $K$-primes is an example of an explicit $\T$-variety in the sense of \cite[Definition 3.8]{MR4031105}.
\end{remark}

We recall the notion of \emph{maximal orbit quotient}, as presented in \cite [Definition 3.12]{MR4031105}:

\begin{definition}\label{moq def}
    Let $X$ be an algebraic variety with an effective action of an algebraic torus $\T\times X\rightarrow X$, and $X_0\subset X$ the open subset consisting of all points $x\in X$ with finite isotropy group. A \emph{maximal orbit quotient} over a variety $Y$ for the $\T-$action on $X$ is a rational quotient $\pi \colon X \dashrightarrow Y$ admitting a representative $\psi\colon W\rightarrow V$ and prime divisors $C_0,\dots,C_r$ on $Y$ such that the following properties are satisfied:
    \begin{itemize}
        \item one has $W\subset X_0$ and the complements $X_0\setminus W\subset X_0$ and $Y\setminus V\subset Y$ are both of codimension two,
        \item for every $i=0,\dots,r$ the inverse image $\psi^{-1}(C_i)\subset W$ is a union of prime divisors $D_{i1},\dots,D_{in_i}\subset W$,
        \item all $\T-$invariant prime divisors of $X_0$ with nontrivial generic isotropy occur among the $D_{ij}$,
        \item every sequence $J=(j_0,\dots,j_r)$ with $1\leq j_i\leq n_i$ defines a geometric quotient $\psi\colon W_J\rightarrow V$ for the $\T-$action, where $W_J\coloneqq W\setminus \bigcup_{j\neq j_i}D_{ij}$.
    \end{itemize}
\end{definition}

\begin{example}
    Consider the $\K^*$-threefold $X$, the open subset $X_0\subset X$, the map $\pi\colon X_0\rightarrow \Pp_2= X/\K^*$ and the prime divisors $C_1,\dots,C_4\subset\Pp_2$ as presented in Example \ref{running example 1}. These data satisfy the properties of Definition \ref{moq def}, hence determining a maximal orbit quotient over $\Pp_2$ for the $\K^*$-action on $X$.
\end{example}

More generally, from \cite[Proposition 3.16]{MR4031105} we deduce the following:
\begin{proposition} \label{T-conditions prop 2}
    Let $X\coloneqq X(g,P)$ be an explicit $\T$-variety arising from Construction \ref{p2cons T1}. The downward rational maps in the commutative diagram
     \[\xymatrix@=35pt{
    X(g,P)\ar@{-->}[d] \ar@{^{(}->}[r] & Z(P)\ar@{-->}[d]\\
    \Pp_2\cong Y(g,P)\ar@{^{(}->}[r] & \Pp(1,1,1,d_3,\dots,d_n)
    }\]
    are maximal orbit quotients.
\end{proposition}

\begin{remark}
    The maximal orbit quotients of Proposition \ref{T-conditions prop 2} are surjective, and they coincide with the Chow quotients of the respective $\T$-varieties. Indeed, let the fan $\Sigma(P)$, the toric variety $Z(P)$ and the $\T$-variety $X=X(g,P)$ arise from Construction \ref{p2cons T1}. Assume first that $X$ is of type A, that is $P$ arises from Construction \ref{cons A}.
    Then
    $$Z(P)_0\coloneqq Z(P)\setminus Z(P)^{\C^*}=\bigsqcup_{\sigma\in\Sigma_0} \mathcal{O}(\sigma),$$
    where $\Sigma_0\coloneqq\Sigma(P)\setminus (\{ \sigma \in \Sigma(P) \mid \{v_1\} \prec \sigma \}\cup \{v_{01},\dots,v_{n1}\})$.
    The maximal orbit quotient for the $\K^*$-action on $Z(P)$ admits a geometric representative $Z(P)_0\rightarrow \Pp(1,1,1,d_3,\dots,d_n)$, sending $\mathcal{O}(\{v_{i_11},\dots,v_{i_k1}\})$ to $\mathcal{O}(\{v_{i_1},\dots,v_{i_k}\})$.
    We then have the commutative diagram
    \[\xymatrix@=35pt{
    X_0\coloneqq X\cap Z(P)_0\ar[d] \ar@{^{(}->}[r] & Z(P)_0\ar[d]\\
    \Pp_2\cong Y(g,P)\ar@{^{(}->}[r] & \Pp(1,1,1,d_3,\dots,d_n),
    }\]
    where the map $X_0\rightarrow Y(g,P)$ is a geometric representative for the maximal orbit quotient of $X$.
    Now assume that $X$ is of type B, that is $P$ arises from Construction \ref{cons B}.
    Let
    $$    Z_1\coloneqq\bigsqcup_{\sigma\in\Sigma_1} \mathcal{O}(\sigma),$$
    where $\Sigma_1\coloneqq\Sigma(P)\setminus (\{ \sigma \in \Sigma(P) \mid \{v_{02}\} \prec \sigma \}\cup \{v_{01},\dots,v_{n1}\})$. Define $Z_2$ analogously.
    Then, as above, for $i=1,2$ we get commutative diagrams
    \[\xymatrix@=35pt{
    W_i\coloneqq X\cap Z_i\ar[d] \ar@{^{(}->}[r] & Z_i\ar[d]\\
    \Pp_2\cong Y(g,P)\ar@{^{(}->}[r] & \Pp(1,1,1,d_3,\dots,d_n),
    }\]
    where the downward maps are geometric representatives for the maximal orbit quotients of the $\T$-varieties $X$ and $Z(P)$ respectively.
    Type C follows analogously.
\end{remark}

\begin{lem} \label{unique moq}
    Let $X$, $X'$ be $\T$-varieties, $\phi:X\rightarrow X'$ an equivariant isomorphism. Assume that $X$ and $X'$ admit maximal orbit quotients $\pi:X\dashrightarrow Y$, $\pi':X'\dashrightarrow Y'$ respectively. Then we have a commutative diagram
    \[\xymatrix@=35pt{
    X\ar@{-->}[d]_{\pi} \ar[r]^\phi & X'\ar@{-->}[d]^{\pi'}\\
    Y\ar@{-->}[r]^\psi & Y'
    }\]
    where $\psi\colon Y\dashrightarrow Y'$ is an isomorphism in codimension two.
\end{lem}

\begin{proof}
    Let $X_0\subset X$, $X_0'\subset X'$ be the open subsets consisting of all points with finite isotropy. Then $\phi_{|X_0}\colon X_0\rightarrow X_0'$ is also an equivariant isomorphism. By Definition \ref{moq def}, $\pi$ and $\pi'$ admit representatives $\pi\colon W\rightarrow V$, $\pi'\colon W'\rightarrow V'$, where $W\subset X_0$, $W'\subset X_0'$ and the complements $X_0\setminus W\subset X_0$, $X_0'\setminus W'\subset X_0'$ and $Y\setminus V\subset Y$, $Y'\setminus V'\subset Y'$ have codimension two. Up to appropriately shrinking $W$ and $W'$, we can assume that $\phi_{|W}\colon W\rightarrow W'$ is also an equivariant isomorphism. Hence there exists an isomorphism $\psi\colon V\rightarrow V'$ that fits in the following commutative diagram:
    \[\xymatrix@=35pt{
    W\ar[d]_{\pi} \ar[r]^{\phi_{|W}} & W'\ar[d]^{\pi'}\\
    V\ar[r]^\psi & V'
    }\]
\end{proof}

\begin{proposition} \label{we get all}
    Let $X$ be a $\Q$-factorial, projective $\K^*$-threefold of Picard number one with a maximal orbit quotient over $\Pp_2$. Then there exists a $\K^*$-equivariant isomorphism $\phi:X\rightarrow X'$, where $X'$ is an explicit $\T$-variety arising form Construction \ref{p2cons T1}.
\end{proposition}

\begin{proof}
    By \cite[Theorem 3.10]{MR4031105}, there exists a $\K^*$-equivariant isomorphism $\phi:X\rightarrow X'$, where $X'$ is an explicit $\T$-variety arising from \cite[Construction 3.5]{MR4031105}.
    Then we have a commutative diagram
    \[\xymatrix@=35pt{
    X\ar@{-->}[d] \ar[r]^\phi & X'\ar@{-->}[d]\ar@{^{(}->}[r] & Z(P)\ar@{-->}[d]\\
    \Pp_2\ar@{-->}[r]^\psi & Y'\ar@{^{(}->}[r] & Z(B)
    }\]
    where the vertical maps are maximal orbit quotients, $Z(B)$ and $Z(P)$ are toric varieties with generator matrix $B$, $P$ respectively, and $\psi\colon Y\dashrightarrow Y'$ is an isomorphism in codimension two, whose existence is guaranteed by Lemma \ref{unique moq}.
    By Proposition \ref{p2 explicit}, it follows that $Y'=Y(g,d)\subset\Pp(1,1,1,d)$ arises from Construction \ref{p2cons T0} and $B=[u_0,\dots,u_r]$ is the generator matrix of $\Pp(1,1,1,d)$. Then $P$ is an $(r+s)\times(n+m)$ integer matrix of the form
    $$
    P=[v_{ij},v_k]=
    \begin{bmatrix}
    l_{01}u_0 & \dots & l_{0n_0}u_0 & \dots & l_{r1}u_r & \dots & l_{rn_r}u_r & 0 & \dots & 0\\
    d_{01} & \dots & d_{0n_0} & \dots & d_{r1} & \dots & d_{rn_r} & d_1' & \dots & d_m'
    \end{bmatrix},
    $$
    where
    \begin{itemize}
        \item $n_0,\dots,n_r\in\Z_{>0}$ and $m,s\in \Z_{\geq 0}$, with $n\coloneqq n_0+\dots+n_r$;
        \item $l_{ij}\in \Z_{>0}$, $d_{ij}\in\Z^s$ for $i=0,\dots,r$, $j=1,\dots,n_i$;
        \item $d_1',\dots, d_m'\in\Z^s$,
        \item the vectors $v_{ij},v_k\in\Z^{n+m}$ are primitive, pairwise distinct and generate $\Q^{t+s}$ as a vector space.
    \end{itemize}
    Since $X'$ is projective, we can further assume that $v_{ij},v_k\in\Z^{n+m}$ generate $\Q^{t+s}$ as a cone.
    By \cite[Corollary 3.8]{MR4031105} we have $\dim(X')=\dim(Y)+s$ and $\Cl(X')=\Z^{n+m}/\im(P^*)$. Since $\dim(X')=3$ and $\dim(Y)=2$, it follows that $s=1$. Furthermore 
    $$1=\rank(\Cl(X'))=\rank(\Z^{n+m}/\im(P^*))=(n_0-1)+\dots+(n_r-1)+m,$$ hence either 
    $m=n_0=\dots=n_r=1$ (Case A) or $n_i=2$ for a certain $i=0,\dots,r$, $n_j=1$ for $j\neq i$ and $m=0$ (Case B or C).
    We conclude that $X'$ is an explicit $\T$-variety arising from Construction \ref{p2cons T1}.
\end{proof}

\section{Gorenstein Fano threefolds with a $\K^*$-action} \label{classification section}

We determine combinatorial criterions for the non-toric, Fano and Gorenstein conditions of an explicit $\T$-variety $X(g,P)$ arising from Construction \ref{p2cons T1}, see Remark \ref{non-toricity}, Propositions \ref{fanoprop} and $\ref{gorensteincor2}$. These conditions involve only the matrix $P$, leaving us some freedom in the choice of the polynomial vector $g$.

\begin{convention} \label{weight conv}
    Consider $K=\Z^{n+2}/\im(P^*)$ and $\omega_{ij}$, $\omega_1$ as in Constructions \ref{cons A}, \ref{cons B} and \ref{cons C}. As $K$ is of rank one, we can fix an isomorphism
    $$
    K\rightarrow \Z\oplus\bigoplus_{i=1}^{q}\Z/a_i\Z,\hspace{10pt} \omega\mapsto (w,\eta),
    $$
    with $q,a_1,\dots,a_q\in\Z_{>0}$ and such that $\omega_{ij}$, $\omega_1$ are mapped to $(w_{ij},\eta_{ij})$, $(w_1,\eta_1)$ respectively, with $w_{01}>0$.
    
\end{convention}

\begin{definition}\label{fwv}
    Let $P=[v_0,\dots,v_n]$ be a generator matrix and consider the integers
    $$
    \tilde{w}_i\coloneqq |det(v_j; j\neq i)|,\quad i=0,\dots,n, \quad \mu\coloneqq\gcd(w_0,\dots,w_n).
    $$
    Then the \emph{fake weight vector} of $P$ is $\tilde{w}(P)=(\tilde{w}_0,\dots,\tilde{w}_n)$ and the \emph{reduced fake weight vector} of $P$ is $w(P)\coloneqq\mu^{-1}\tilde{w}$.
\end{definition}

\begin{lem} \label{weight lem}
    Let $P$ arise from Construction \ref{cons A}, \ref{cons B} or  \ref{cons C}. According to the type of $P$, the fake weight vector $\tilde{w}\coloneqq\tilde{w}(P)$ is given by:
    \newline
    {\footnotesize
    \begin{itemize}
        \item[A:]
            \begin{itemize}
                \item[] $\tilde{w}_{i1}=l_{01}\cdots l_{n1}/l_{i1}$ for $i=0,1,2$,
                \item[] $\tilde{w}_{i1}=d_{i}l_{01}\cdots l_{n1}/l_{i1}$ for $i\geq 3$,
                \item[] $\tilde{w}_1=l_{01}\cdots l_{n1}\left(\dfrac{d_{01}}{l_{01}}-\dfrac{d_{11}}{l_{11}}-\dfrac{d_{21}}{l_{21}}-\sum_{j=3}^n d_j\dfrac{d_{j1}}{l_{j1}}\right)$.
            \end{itemize}
        \item[B:]
            \begin{itemize}
                \item[] $\tilde{w}_{01}=l_{02}l_{11}\cdots l_{n1}\left(\dfrac{d_{02}}{l_{02}}-\dfrac{d_{11}}{l_{11}}-\dfrac{d_{21}}{l_{21}}-\dfrac{d_3d_{31}}{l_{31}}-\dots-\dfrac{d_nd_{n1}}{l_{n1}}\right)$,
                \item[] $\tilde{w}_{02}=l_{01}l_{11}\cdots l_{n1}\left(\dfrac{d_{11}}{l_{11}}+\dfrac{d_{21}}{l_{21}}+\dfrac{d_3d_{31}}{l_{31}}+\dots+\dfrac{d_nd_{n1}}{l_{n1}}-\dfrac{d_{01}}{l_{01}}\right)$,
                \item[] $\tilde{w}_{i1}=\dfrac{d_il_{11}\cdots l_{n1}}{l_{i1}}(l_{01}d_{02}-l_{02}d_{01})$ for $i\in\{1\dots,n\}$.
            \end{itemize}
        \item[C:]
            \begin{itemize}
                \item[] $\tilde{w}_{i1}=\dfrac{d_il_{11}\cdots l_{n1}}{l_{i1}}(l_{31}d_{32}-l_{32}d_{31})$ for $i\in\{0\dots,n\}\setminus\{3\}$,
                \item[] $\tilde{w}_{31}=l_{01}l_{11}l_{21}l_{32}l_{41}\cdots l_{n1}\left(\dfrac{d_3d_{32}}{l_{32}}-\dfrac{d_{11}}{l_{11}}-\dfrac{d_{21}}{l_{21}}-\dfrac{d_4d_{41}}{l_{41}}-\dots-\dfrac{d_nd_{n1}}{l_{n1}}\right)$,
                \item[] $\tilde{w}_{32}=l_{01}l_{11}l_{21}l_{31}l_{41}\cdots l_{n1}\left(\dfrac{d_{11}}{l_{11}}+\dfrac{d_{21}}{l_{21}}+\dfrac{d_4d_{41}}{l_{41}}+\dots+\dfrac{d_nd_{n1}}{l_{n1}}-\dfrac{d_3d_{31}}{l_{31}}\right)$.
            \end{itemize}
    \end{itemize}
    }
    \noindent Furthermore, with the notation of Convention \ref{weight conv}, the reduced fake weight vector $w\coloneqq w(P)$ of $Z(P)$ coincide with:
    \begin{itemize}
        \item[A:] $(w_{01},\dots,w_{n1},w_1)$;
        \item[B:] $(w_{01},w_{02},w_{11},\dots,w_{n1})$;
        \item[C:] $(w_{01},\dots,w_{31},w_{32},w_{41},\dots, w_{n1})$.
    \end{itemize}
\end{lem}

\begin{proof}
 The integers $\tilde{w}_{ij}$, $\tilde{w}_1$ are obtained by computing the maximal minors of $P$.
 We prove the second statement for Type A (B and C follow analogously).
    By Convention \ref{weight conv} and Definition \ref{fwv}, both the vector $(w_{01},\dots,w_{n1},w_1)$ and the reduced fake weight vector $w(P)$ are orthogonal to the rows of $P$, hence
    $$
    (w_{01},\dots,w_{n1},w_1)=kw(P)\text{~for~} k\in\Q^*.
    $$ Since both are primitive, $w_{01}>0$ and $\mu^{-1}\tilde{w}_{01}>0$, it follows that $k=1$.
\end{proof}

\begin{remark}\label{fanormk}
Let $X=X(g,P)$ be an explicit $\T$-variety arising from Construction \ref{p2cons T1}. By \cite[Proposition 3.3.3.2]{MR3307753}, the anticanonical class of $X$ is given, according to type, by
\begin{itemize}
    \item[A:] $-\mathcal{K}_X=\sum_{i=0}^2\deg(T_{i1})+\sum_{i=3}^{n}(1-l_{i1})\deg(T_{i1})+\deg(S_1)$;
    \item[B:] $-\mathcal{K}_X=\deg(T_{01})+\deg(T_{02})+\deg(T_{11})+\deg(T_{21})+\sum_{i=3}^{n}(1-l_{i1})\deg(T_{i1})$;
    \item[C:] $-\mathcal{K}_X=\deg(T_{01})+\deg(T_{11})+\deg(T_{21})+(1-l_{32})\det(T_{32})+\vspace{4pt} \\ \sum_{i=3}^{n}(1-l_{i1})\deg(T_{i1}).$
\end{itemize}

\end{remark}

A complete variety is called \textbf{Fano} if the anticanonical class is ample.

\begin{proposition}\label{fanoprop}
Let $X=X(g,P)$ be an explicit $\T$-variety arising from Construction \ref{p2cons T1}. Then, according to type, $X$ is Fano if and only if
\begin{itemize}
    \item[A:] $\dfrac{d_{01}+1}{l_{01}}-\sum_{i=0}^2\dfrac{d_{i1}-1}{l_{i1}}-\sum_{i=3}^n \dfrac{d_i(d_{i1}-1)}{l_{i1}}-\sum_{i=3}^n d_i>0$;
    \item[B:] $\dfrac{d_{02}-d_{01}}{l_{01}d_{02}-l_{02}d_{01}}-\dfrac{l_{02}-l_{01}}{l_{01}d_{02}-l_{02}d_{01}}\left(\sum_{i=1}^n \dfrac{d_id_{i1}}{l_{i1}}\right)+\sum_{i=1}^n\dfrac{d_i}{l_{i1}}-\sum_{i=3}^n d_i>0$;
    \item[C:] $d_3\dfrac{d_{32}-d_{31}}{l_{31}d_{32}-l_{32}d_{31}}-\dfrac{l_{32}-l_{31}}{l_{31}d_{32}-l_{32}d_{31}}\left(\dfrac{d_{01}}{l_{01}}-\dfrac{d_{11}}{l_{11}}-\dfrac{d_{21}}{l_{21}}-\sum_{i=4}^n \dfrac{d_id_{i1}}{l_{i1}}\right)+\vspace{5pt} \\ \dfrac{1}{l_{01}}+
    \dfrac{1}{l_{11}}+\dfrac{1}{l_{21}}+\sum_{i=4}^n\dfrac{d_i}{l_{i1}}-\sum_{i=3}^n d_i>0$.
\end{itemize}
\end{proposition}

\begin{proof}
We prove the proposition for Type A (B and C follow analogously). Since $X$ has Picard number one, it is Fano if and only if the free part of $-\mathcal{K}_X$ is positive.
By Lemma \ref{weight lem}, it follows that $X$ is Fano if and only if
$$
\sum_{i=0}^2 \tilde{w}_{i1}+\sum_{i=3}^n(1-l_{i1})\tilde{w}_{i1}+\tilde{w}_1>0,
$$
and we get the thesis.
\end{proof}

\begin{definition}
    Let $X=X(g,P)$ arise form Construction \ref{p2cons T1}, $\Sigma(P)$ and $Z(P)$ as in Reminder \ref{fwps rem}. We call $\sigma\in\Sigma(P)$ an $X$-cone if the corresponding toric orbit $\T^{n+1}\cdot z_\sigma$ meets $X\subset Z(P)$.
\end{definition}

\begin{proposition}\label{cartierprop}
Let $X=X(g,P)$ arise form Construction \ref{p2cons T1}, $\Sigma(P)$ and $Z(P)$ as in Reminder \ref{fwps rem}. Consider on $X$ a divisor class
$$
w=\sum_{i,j} a_{ij}\omega_{ij}+a_1\omega_1
$$
with $a_{ij},a_1\in\Q$ and an $X$-cone $\sigma\in\Sigma(P)$.
Then $w$ is Cartier on $X_\sigma$ if and only if there exists $u\in\Z^{n+1}$ with $\langle u,v_{ij}\rangle=a_{ij}$ and $\langle u,v_1\rangle=a_1$ for all $v_{ij},v_1\in\sigma$.
\end{proposition}

\begin{proof}
Recall that the divisor $w$ is Cartier if it's the intersection of $X$ with a Cartier divisor $w^Z$ of the ambient toric variety $Z(P)$. On every cone $\sigma\in\Sigma(P)$, $w^Z$ is determined by a character $u\in\Z^{n+1}$, and it is principal on $\sigma$ if and only if $\langle u,v_{i1}\rangle=a_{i1}$ and $\langle u,v_1\rangle=a_1$ for every $v_{i1},v_1\in\sigma$.
\end{proof}

A variety is called \textbf{Gorenstein} if its canonical class is Cartier. Combining Remark \ref{fanormk} and Proposition \ref{cartierprop} we obtain the following characterisation.

\begin{corollary}\label{gorensteincor}
Let $X=X(g,P)$ be an explicit $\T$-variety arising from Construction \ref{p2cons T1}.
Then, according to type, $X$ is Gorenstein if and only if for every maximal $X$-cone $\sigma$ there is a linear form $u\in\Z^{n+1}$ with
\begin{itemize}
    \item[A:] $\langle u,v_{i1}\rangle=1 \text{~for~}i=0,1,2,\hspace{10pt} \langle u,v_{i1}\rangle=1-l_{i1} \text{~for~}i\geq 3, \hspace{10pt} \langle u,v_1\rangle=1$,
    \item[B:] $\langle u,v_{ij}\rangle=1 \text{~for~}i=0,1,2,\hspace{10pt} \langle u,v_{i1}\rangle=1-l_{i1} \text{~for~}i=3,\dots,n$,
    \item[C:] $\langle u,v_{i1}\rangle=1 \text{~for~}i=0,1,2,\hspace{10pt} \langle u,v_{i1}\rangle=1-l_{i1} \text{~for~}i\geq 3, \hspace{10pt} \langle u,v_{32}\rangle=1-l_{32}$,
\end{itemize}
    for all $v_{ij},v_1\in\sigma$.
\end{corollary}

\begin{proposition}\label{X-cones}
Let $X=X(g,P)$ arise from Construction \ref{p2cons T1}. Then, according to type, each of the following 
is an $X$-cone:
\begin{itemize}
\item[A:] $\{v_{01},\dots,v_{n1}\},\hspace{10pt}\{v_{i1},v_1\} \text{~for~}i=0,\dots,n$;
\item[B:] $\{v_{01},v_{11},\dots,v_{n1}\},\hspace{7pt},\{v_{02},v_{11},\dots,v_{n1}\},\hspace{7pt}\{v_{01},v_{02},v_{i1}\} \text{~for~}i=1,\dots,n$;
\item[C:] $\{v_{01},\dots,\widehat{v_{32}},\dots,v_{n1}\},\hspace{7pt}\{v_{01},\dots,\widehat{v_{31}},\dots,v_{n1}\},\hspace{7pt}\{v_{31},v_{32},v_{i1}\} \text{~for~}i\neq 3.$
\end{itemize} 
\end{proposition}

\begin{proof}
    We prove the proposition for Type A (B and C follow analogously).
    First, notice that
    $$
    \Spec(\Rr(X))\cap V(T_{01},\dots,T_{n1})\cong \K,
    $$
    hence $\{v_{01},\dots,v_{n1}\}$ is an $X$-cone. 
    Analogously, for $i=0,\dots,n$ we have
    $$
    \dim(\Spec(\Rr(X))\cap V(T_{i1},T_1))\geq \dim(\Spec(\Rr(X)))-2=2,
    $$
    hence $\{v_{i1},v_1\}$ is an $X$-cone.
\end{proof}

\begin{proposition} \label{maximal cones}
    Let $X=X(g,P)$ arise from Construction \ref{p2cons T1}, $\Sigma(P)$ as in Reminder \ref{fwps rem}. Then every $X$-cone is the face of one of the following cones:
\begin{itemize}
    \item[A:] $\{v_{01},\dots,v_{n1}\},\hspace{5pt}\{v_{01},\dots,v_{n1},v_{1}\}\setminus\{v_{i1}\}$ for $i=0,1,2$;
    \item[B:] $\{v_{01},v_{11},\dots,v_{n1}\},\hspace{5pt}\{v_{02},v_{11},\dots,v_{n1}\},\hspace{5pt}
    \{v_{01},v_{02},v_{11},\dots,v_{n1}\}\setminus\{v_{i1}\}$ for $i=1,2$;
    \item[C:] $\{v_{01},\dots,\widehat{v_{32}},\dots,v_{n1}\},\hspace{5pt}
    \{v_{01},\dots,v_{n1}\}\setminus\{v_{i1}\}$ for $i=1,2,3$.
\end{itemize}
\end{proposition}

\begin{proof}
    We prove the proposition for Type A (B and C follow analogously). Since
    $$
    \Spec(\Rr(X))\cap V(T_{01},T_{11},T_{21},S_1)=\{(0,\dots,0)\},
    $$
    any cone containing $\{v_{01},v_{11},v_{21},v_{1}\}$ is not an $X$-cone. Therefore, every $X$-cone is a face of a maximal cone of $\Sigma(P)$ not containing $\{v_{01},v_{11},v_{21},v_{1}\}$.
\end{proof}

\begin{remark}\label{non-toricity}
    Let an explicit $\T$-variety $X=X(g,P)$ and $h_3,\dots,h_n$ arise from Construction \ref{p2cons T1}. Then, according to type, the following sets of conditions are equivalent to $h_3,\dots,h_n$ having no linear term:
    \begin{itemize}
        \item[A:]
            \begin{itemize}
            \item $l_{31},\dots,l_{n1}\geq 2$,
            \item if $d_i=1$, i.e. $g_i=\alpha_0T_0+\alpha_1T_1+\alpha_2T_2+T_i$, then $l_{j1}\geq 2$ for every $j$ such that $\alpha_j\neq 0$;
            \end{itemize}
        \item[B:]
            \begin{itemize}
            \item $l_{31},\dots,l_{n1}\geq 2$,
            \item if $d_i=1$, i.e. $g_i=\alpha_0T_0+\alpha_1T_1+\alpha_2T_2+T_i$, then $l_{11}\geq 2$ if $\alpha_1\neq 0$ and $l_{21}\geq 2$ if $\alpha_2\neq 0$;
            \end{itemize}
        \item[C:]
            \begin{itemize}
            \item $l_{41},\dots,l_{n1}\geq 2$,
            \item if $d_i=1$, i.e. $g_i=\alpha_0T_0+\alpha_1T_1+\alpha_2T_2+T_i$, then $l_{j1}\geq 2$ for every $\alpha_j\neq 0$.
            \end{itemize}
    \end{itemize}
    Recall that $X$ is a toric variety if and only if its Cox ring $\Rr(X)$ is a polynomial ring.
   Hence, the above sets of conditions imply that $X$ is non-toric. Furthermore, any non-toric variety $X(g,P)$ arising from Construction \ref{p2cons T1} is $\K^*$-isomorphic to a variety $X(g',P')$ arising from Construction \ref{p2cons T1} and satisfying the above conditions.
\end{remark}

\begin{remark}\label{non-redundancy}
    Let $X\coloneqq X(g,P)$, $X'\coloneqq X(g',P')$, $h_i$ and $h_i'$ arise from Construction \ref{p2cons T1} with Type A.
    If there exist an $n\times n$ unimodular matrix $U$ and a permutation $\pi\in S_n$ with associated permutation matrix $S_{\pi}$ such that
    \begin{itemize}
        \item $UPS_{\pi}=P'$,
        \item $h_i'(T_{01},T_{11},T_{21},T_{i1})=h_i(T_{\pi^{-1}(0)1},T_{\pi^{-1}(1)1},T_{\pi^{-1}(2)1},T_{\pi^{-1}(i)1})$ for every $i$,
    \end{itemize}
    then $X$ and $X'$ are isomorphic. We can get analogous conditions for Type B and C.
    In particular, without loss of generality we will assume that:
    \begin{itemize}
        \item[A:]
        \begin{itemize}
            \item $l_{i1}>d_{i1}\geq 0$ for $i=1,\dots,n$,
            \item $l_{21}\geq l_{11}\geq l_{01}$,
            \item $d_n\geq\dots\geq d_3$,
            \item if $d_m=d_{m+1}=\dots=d_{m+j}$ then $l_{(m+j)1}\geq\dots\geq l_{m1}$ for $m\geq 3$.
        \end{itemize}
        \item[B:]
        \begin{itemize}
            \item $l_{i1}>d_{i1}\geq 0$ for $i=1,\dots,n$,
            \item $l_{21}\geq l_{11}$,
            \item $l_{02}\geq l_{01}$,
            \item $d_n\geq\dots\geq d_3$,
            \item if $d_m=d_{m+1}=\dots=d_{m+j}$ then $l_{(m+j)1}\geq\dots\geq l_{m1}$ for $m\geq 3$.
        \end{itemize}
        \item[C:]
        \begin{itemize}
            \item $l_{i1}>d_{i1}\geq 0$ for $i=1,\dots,n$,
            \item $l_{01}\geq l_{21}\geq l_{11}$,
            \item $l_{32}\geq l_{31}$,
            \item $d_n\geq\dots\geq d_4$,
            \item if $d_m=d_{m+1}=\dots=d_{m+j}$ then $l_{(m+j)1}\geq\dots\geq l_{m1}$ for $m\geq 4$.
        \end{itemize}
    
    \end{itemize}
\end{remark}

\begin{proposition}\label{gorensteincor2}
Let $X=X(g,P)$ be a non-toric explicit $\T$-variety arising form Construction \ref{p2cons T1}. Then, according to type, $X$ is a Gorenstein variety if and only if the following conditions are satisfied:
\begin{itemize}
    \item[A:]
        \begin{itemize}
            \item $\dfrac{1+d_{01}}{{l_{01}}}\in\Z_{>0}$;
            \item $d_{11}=\min(1,l_{11}-1)$, $d_{21}=\min(1,l_{21}-1)$, $d_{j1}=1$ for $i\geq 3$;
            \item $\alpha\coloneqq\dfrac{\dfrac{1+d_{01}}{l_{01}}+\dfrac{1-d_{11}}{l_{11}}+\dfrac{1-d_{21}}{l_{21}}-\sum_{j=3}^n d_j}{\dfrac{d_{01}}{l_{01}}-\dfrac{d_{11}}{l_{11}}-\dfrac{d_{21}}{l_{21}}-\sum_{j=3}^n\dfrac{d_j}{l_{j1}}}\in l\Z,$
        \end{itemize}
        where $l\coloneqq \lcm(l_{11},\dots,l_{n1})$;
    \item[B:]
        \begin{itemize}
            \item $\nu\coloneqq\dfrac{d_{02}-d_{01}}{l_{01}d_{02}-l_{02}d_{01}}\in\Z_{>0}$, \hspace{10pt} $\mu\coloneqq\dfrac{l_{02}-l_{01}}{l_{01}d_{02}-l_{02}d_{01}}\in\Z_{>0}$,
            \item $\alpha_j\coloneqq\dfrac{\dfrac{1}{l_{0j}}+\dfrac{1}{l_{11}}+\dfrac{1}{l_{21}}+\dfrac{d_3}{l_{31}}+\sum_{i=4}^n\dfrac{d_i}{l_{i1}}-\sum_{i=3}^n d_i}{\dfrac{d_{01}}{l_{01}}+\dfrac{d_{11}}{l_{11}}+\dfrac{d_{21}}{l_{21}}+\sum_{i=4}^n\dfrac{d_i d_{i1}}{l_{i1}}-\dfrac{d_{0j}}{l_{0j}}}\in \Z$ for $j=1,2$,
            \item $\lambda_i\coloneqq\dfrac{\mu d_{i1}-1}{{l_{i1}}}\in\Z$, $\beta_i\coloneqq\dfrac{\alpha_1 d_{i1}-1}{{l_{i1}}}\in\Z$, $\gamma_i\coloneqq\dfrac{\alpha_2 d_{i1}-1}{{l_{i1}}}\in\Z$ for $i=1,\dots,n$.
        \end{itemize}
    \item[C:]
        \begin{itemize}
            \item $\nu\coloneqq\dfrac{d_{32}-d_{31}}{l_{31}d_{32}-l_{32}d_{31}}\in\Z_{>0}$, \hspace{10pt} $\mu\coloneqq\dfrac{l_{32}-l_{31}}{l_{31}d_{32}-l_{32}d_{31}}\in\Z_{>0}$,
            \item $\alpha_j\coloneqq\dfrac{\dfrac{1}{l_{01}}+\dfrac{1}{l_{11}}+\dfrac{1}{l_{21}}+\dfrac{d_3}{l_{3j}}+\sum_{i=4}^n\dfrac{d_i}{l_{i1}}-\sum_{i=3}^n d_i}{\dfrac{d_3d_{3j}}{l_{3j}}-\left(\dfrac{d_{01}}{l_{01}}-\dfrac{d_{11}}{l_{11}}-\dfrac{d_{21}}{l_{21}}-\sum_{i=4}^n\dfrac{d_i d_{i1}}{l_{i1}}\right)}\in \Z$ for $j=1,2$,
            \item $\lambda_0\coloneqq\dfrac{\mu d_{01}-1}{{l_{01}}}\in\Z$, $\beta_0\coloneqq\dfrac{\alpha_1 d_{01}+1}{{l_{01}}}\in\Z$, $\gamma_0\coloneqq\dfrac{\alpha_2 d_{01}+1}{{l_{01}}}\in\Z$,
            \item $\lambda_i\coloneqq\dfrac{\mu d_{i1}+1}{{l_{i1}}}\in\Z$, $\beta_i\coloneqq\dfrac{\alpha_1 d_{i1}-1}{{l_{i1}}}\in\Z$, $\gamma_i\coloneqq\dfrac{\alpha_2 d_{i1}-1}{{l_{i1}}}\in\Z$ for $i\neq 3$.
    \end{itemize}
\end{itemize}
\end{proposition}

\begin{proof}
We prove the proposition for Type A (B and C follow analogously). We know that $X$ is Gorenstein if and only if it satisfies the conditions of Corollary \ref{gorensteincor} associated to its maximal $X$-cones. Hence, the cones of Proposition \ref{X-cones} provide necessary conditions, and the cones of Proposition \ref{maximal cones} provide sufficient conditions.  We write $u=(u_1,\dots,u_{n+1})\in\Z^{n+1}$.
The condition associated to the $X$-cone $\{v_{11},v_{1}\}$ is:
$$
l_{11}u_1+d_{11}u_{n+1}=1, \hspace{5pt} u_{n+1}=1 \implies \dfrac{1-d_{11}}{l_{11}}\in\Z.
$$
If $l_{11}=1$ then $d_{11}=0$, since $l_{11}>d_{11}\geq 0$. If $l_{11}\geq 2$ then $d_{11}=1$, since $l_{11}$ must divide $1-d_{11}$, but $|1-d_{11}|<l_{11}$. Hence, we get the condition
$$d_{11}=\min(1,l_{11}-1).$$
Analogously, the conditions associated with the $X$-cones $\{v_{01},v_{1}\}$, $\{v_{21},v_{1}\}$, $\{v_{j1},v_{1}\}$ for $j=3,\dots,n$ are respectively
\begin{itemize}
\item $-l_{01}(u_1+u_2+\sum_{i=3}^n d_i u_i)-d_{01}=1, \hspace{5pt} u_{n+1}=1 \implies \dfrac{1+d_{01}}{l_{01}}\in\Z_{>0}$,
\item $l_{21}u_2+d_{21}u_{n+1}=1, \hspace{5pt} u_{n+1}=1 \implies d_{21}=\min(1,l_{21}-1)$,
\item $l_{j1}u_j+d_{j1}u_{n+1}=1-l_{j1}, \hspace{5pt} u_{n+1}=1  \implies d_{j1}=1$.
\end{itemize}

Then the condition associated with the $X$-cone $\{v_{01},\dots,v_{n1}\}$ becomes:
\begin{align*}
&u_1=\dfrac{1-d_{11}u_{n+1}}{l_{11}}\in\Z, \hspace{10pt} u_2=\dfrac{1-d_{21}u_{n+1}}{l_{21}}\in\Z,\\
&u_j=\dfrac{1-u_{n+1}}{l_{j1}}-1\in\Z, \hspace{10pt} -l_{01}\left(u_1+u_2+\sum_{j=3}^n d_j u_j\right)-d_{01}u_{n+1}=1,
\end{align*}
where $j=3,\dots,n$ and we used the fact that $d_{j1}=1$.
By substituting the first three identities in the fourth one and factoring out $u_{n+1}$ we get
$$
u_{n+1}\left(\dfrac{d_{01}}{l_{01}}-\dfrac{d_{11}}{l_{11}}-\dfrac{d_{21}}{l_{21}}-\sum_{j=3}^n\dfrac{d_j}{l_{j1}}\right)=\sum_{j=3}^n d_j-\dfrac{1}{l_{01}}-\dfrac{1}{l_{11}}-\dfrac{1}{l_{21}}-\sum_{j=3}^n\dfrac{d_j}{l_{j1}}.
$$
From Construction \ref{p2cons T1} with $d_{j1}=1$ for $j\geq 3$ we have that
$$
\dfrac{d_{01}}{l_{01}}-\dfrac{d_{11}}{l_{11}}-\dfrac{d_{21}}{l_{21}}-\sum_{j=3}^n\dfrac{d_j}{l_{j1}}>0,
$$
so we can divide by it on both sides and obtain
$$
u_{n+1}=\dfrac{\sum_{j=3}^n d_j-\left(\dfrac{1}{l_{01}}+\dfrac{1}{l_{11}}+\dfrac{1}{l_{21}}+\sum_{j=3}^n\dfrac{d_j}{l_{j1}}\right)}{\dfrac{d_{01}}{l_{01}}-\dfrac{d_{11}}{l_{11}}-\dfrac{d_{21}}{l_{21}}-\sum_{j=3}^n\dfrac{d_j}{l_{j1}}}\in\Z,
$$
Furthermore, from $l_{11}|1-d_{11}$ and $l_{11}|1-d_{11}u_{n+1}$ it follows that $l_{11}|d_{11}(1-u_{n+1})$. Since $\gcd(d_{11},l_{11})=1$, we have $l_{11}|1-u_{n+1}$. Analogously, $l_{j1}|1-u_{n+1}$ for $j\geq 1$. Hence, $X$ satisfies the condition associated with the $X$-cone $\{v_{01},\dots,v_{n1}\}$ if and only if
$$
1-u_{n+1}=\dfrac{\dfrac{1+d_{01}}{l_{01}}+\dfrac{1-d_{11}}{l_{11}}+\dfrac{1-d_{21}}{l_{21}}-\sum_{j=3}^n d_j}{\dfrac{d_{01}}{l_{01}}-\dfrac{d_{11}}{l_{11}}-\dfrac{d_{21}}{l_{21}}-\sum_{j=3}^n\dfrac{d_j}{l_{j1}}}\in \bigcap_{i=1}^n l_{i1}\Z=l\Z,
$$
where $l\coloneqq \lcm(l_{11},\dots,l_{n1})$.

Now, the conditions associated to the maximal cone $\{v_{01},v_{11},v_{31}\dots,v_{n1},v_{1}\}$ are
\begin{itemize}
\item $-l_{01}\left(u_1+u_2+\sum_{j=3}^nd_ju_j\right)-d_{01}=1$,
\item $l_{11}u_1+d_{11}u_{n+1}=1$,
\item $l_{j1}u_j+d_{j1}u_{n+1}=1-l_{31}$,
\item $u_{n+1}=1$,
\end{itemize}
which we rewrite as
\begin{align*}
&u_1=\dfrac{1-d_{11}}{l_{11}}\in\Z,\hspace{5pt} u_j=\dfrac{1-d_{j1}}{l_{j1}}-1\in\Z, \hspace{5pt} u_2=-u_1-\sum_{j=3}^n du_j+\dfrac{1+d_{01}}{l_{01}}\in\Z,\\
&\implies \dfrac{1-d_{11}}{l_{11}}\in\Z,\hspace{5pt} \dfrac{1-d_{j1}}{l_{j1}}-1\in\Z, \hspace{5pt} \dfrac{1+d_{01}}{l_{01}}\in\Z_{>0},
\end{align*}
and are therefore already implied by the necessary conditions above. The same goes for the other cones of Proposition \ref{maximal cones}. Hence the necessary conditions provided by the cones of Proposition \ref{X-cones} are also sufficient.
\end{proof}

\section{Proof of Theorem \ref{big theorem}}\label{proof of theorem}
We divide the proof of Theorem \ref{big theorem} in four steps: in the first three we determine the families of non-toric Gorenstein Fano threefolds arising from Construction \ref{p2cons T1} of Type A, B and C respectively, by means of Proposition \ref{fanoprop} and \ref{gorensteincor2}. In the last step, we prove that every non-toric, $\Q$-factorial, Gorenstein Fano threefold of Picard number
one with an effective $\K^*$-action and maximal orbit quotient
$\Pp_2$ belongs to one of these 154 families, and that members of different families are not isomorphic as $\K^*$-threefolds. First, we need the following 
\begin{remark} \label{unit fraction remark}
    Fix $a_1,\dots,a_n\in\Z_{>0}$ and $q\in \Q$. We find all positive integer solutions for $x_1,\dots,x_n$ in the equation
    $$
    \dfrac{a_1}{x_1}+\dots+\dfrac{a_n}{x_n}=q
    $$
    as follows. If $n=1$, then $x_1=\dfrac{a_1}{q}$ is the only solution if $\dfrac{a_1}{q}\in\Z_{>0}$; otherwise we have no solutions. If $n\geq 2$, assume first $x_1\geq\dots\geq x_n$. Then
    $$
    x_n\leq \dfrac{a_1+\dots+a_n}{q},
    $$
    and for every $b_n\in\Z_{>0}$ with $b_n\leq \dfrac{a_1+\dots+a_n}{q}$ we follow the procedure recursively for the equation
    $$
    \dfrac{a_1}{x_1}+\dots+\dfrac{a_{n-1}}{x_{n-1}}=q-\dfrac{a_n}{b_n},
    $$
    hence finding all positive integer solutions with $x_1\geq\dots\geq x_n$.
    Then, for every permutation $\pi\in S_n$ we follow the procedure again for the equation
    $$
    \dfrac{a_{\pi(1)}}{x_1}+\dots+\dfrac{a_{\pi(n)}}{x_n}=q,
    $$
    hence finding all positive integer solutions.
\end{remark}

\begin{proof}[Proof of Theorem \ref{big theorem}, Type A]
    By Construction \ref{p2cons T1}, Proposition \ref{fanoprop} and Proposition \ref{gorensteincor2} we have
    \begin{align*}
    \alpha_N\coloneqq &\dfrac{1+d_{01}}{l_{01}}+\dfrac{1-d_{11}}{l_{11}}+\dfrac{1-d_{21}}{l_{21}}-\sum_{j=3}^n d_j\in\Z_{>0};\\
    \alpha_D\coloneqq &\dfrac{d_{01}}{l_{01}}-\dfrac{d_{11}}{l_{11}}-\dfrac{d_{21}}{l_{21}}-\sum_{j=3}^n\dfrac{d_j}{l_{j1}}=\dfrac{1+d_{01}}{l_{01}}-\dfrac{1}{l_{01}}-\dfrac{d_{11}}{l_{11}}-\dfrac{d_{21}}{l_{21}}-\sum_{j=3}^n\dfrac{d_j}{l_{j1}}>0.
    \end{align*}
    Furthermore, $\alpha_N\geq l\alpha_D$, where $l\coloneqq \lcm(l_{11},\dots,l_{n1})$.
    In particular
    \begin{align*}
        0&\leq\alpha_N-2\alpha_D=-\dfrac{1+d_{01}}{l_{01}}+\dfrac{d_{11}}{l_{11}}+\dfrac{d_{21}}{l_{21}}+\dfrac{2}{l_{01}}+\dfrac{1}{l_{11}}+\dfrac{1}{l_{21}}-\sum_{j=3}^n \dfrac{(l_{j1}-2)d_j}{l_{j1}}\\
        &\leq \dfrac{2}{l_{01}}+\dfrac{2}{l_{11}}+\dfrac{2}{l_{21}}-\dfrac{1+d_{01}}{l_{01}}\implies\dfrac{1+d_{01}}{l_{01}}\leq\dfrac{2}{l_{01}}+\dfrac{2}{l_{11}}+\dfrac{2}{l_{21}}\leq6.
    \end{align*}
    Since $\alpha_N\geq 1$ we have
    $$
    6\geq\dfrac{1+d_{01}}{l_{01}}=\alpha_N+\sum_{j=3}^n d_j-\dfrac{1-d_{11}}{l_{11}}-\dfrac{1-d_{21}}{l_{21}}\geq\sum_{j=3}^nd_j-1,
    $$
    hence $\sum_{j=3}^n d_j\leq 7$. In particular, $n\leq 9$.
    Finally, by rearranging $\alpha\in\Z_{>0}$ we obtain the equation
    \begin{align*}
    & \dfrac{1}{l_{01}}+\dfrac{1}{l_{11}}+\dfrac{1}{l_{21}}+\sum_{j=3}^n\dfrac{d_j}{l_{j1}}+\dfrac{\dfrac{d_{01}+1}{l_{01}}-\dfrac{d_{11}-1}{l_{11}}-\dfrac{d_{21}-1}{l_{21}}-\sum_{j=3}^n d_j}{\alpha}=\\
    &\dfrac{d_{01}+1}{l_{01}}-\dfrac{d_{11}-1}{l_{11}}-\dfrac{d_{21}-1}{l_{21}}.
    \end{align*}
    For each possible value of $\dfrac{d_{01}+1}{l_{01}}, \dfrac{d_{11}-1}{l_{11}}, \dfrac{d_{21}-1}{l_{21}}$, $n$ and $\sum_{j=3}^n d_j$ we compute the solutions of the above equation via Remark \ref{unit fraction remark}, and we check which of those satisfy the conditions of Proposition \ref{fanoprop} and \ref{gorensteincor2}. This procedure identifies 50 families of Gorenstein Fano threefolds of Type A arising form Construction \ref{p2cons T1}.
\end{proof}

\begin{lem} \label{cases lemma}
    Let $X=X(g,P)$ be a non-toric Gorenstein Fano threefold of Type B arising form Construction \ref{p2cons T1}. Then
    $$
    \dfrac{1}{l_{11}}+\dfrac{1}{l_{21}}+\sum_{i=3}^n\dfrac{d_i}{l_{i1}}-\sum_{i=3}^n d_i>0.
    $$
    In particular, either $n=3$ and $d_3\leq 3$ or $n=4$ and $d_3=d_4=1$.
\end{lem}

\begin{proof}
    We use the notation of Proposition \ref{gorensteincor2}. By Proposition \ref{fanoprop} and since $\lambda_i\coloneqq\dfrac{\mu d_{i1}-1}{{l_{i1}}}\in\Z$, $\nu=\dfrac{\mu d_{01}+1}{l_{01}}\in\Z_{>0}$ we have
    \begin{align*}
        &\nu-\lambda_1-\lambda_2-\sum_{i=3}^n d_i\lambda_i>\sum_{i=3}^n d_i\implies \nu-\lambda_1-\lambda_2-\sum_{i=3}^n d_i\lambda_i\geq1+\sum_{i=3}^n d_i \implies \\
        & \dfrac{1}{l_{11}}+\dfrac{1}{l_{21}}+\sum_{i=3}^n\dfrac{d_i}{l_{i1}}-\sum_{i=3}^n d_i>1-\dfrac{1}{l_{01}}+\mu\left(\dfrac{d_{11}}{l_ {11}}+\dfrac{d_{21}}{l_{21}}+\sum_{i=3}^n\dfrac{d_id_{i1}}{l_{i1}}-\dfrac{d_{01}}{l_{01}}\right)>0.
    \end{align*}
   
    Since $l_{i1}\geq 2$ for $i\geq 3$, we have
    $$
    \sum_{i=3}^n d_i<\sum_{i=1}^n\dfrac{d_i}{l_{i1}}\leq 2+\dfrac{\sum_{i=3}^n d_i}{2}\implies \sum_{i=3}^n d_i\leq 3.
    $$
    If $n\geq 4$, then $d_i=1$ for some $i\in\{3,\dots,n\}$. By Remark \ref{non-toricity}, it follows that $l_{21}\geq 2$, which implies $\sum_{i=3}^n d_i\leq 2$, and hence $n=4$, $d_3=d_4=1$.
\end{proof}

\begin{corollary}
    Let $X=X(g,P)$ be a non-toric Gorenstein Fano threefold of Type B arising form Construction \ref{p2cons T1}. With the notation of Proposition \ref{gorensteincor2}, we have $\alpha_1\in\Z_{>0}$ and $\alpha_2\in\Z_{<0}$.
\end{corollary}

\begin{proof}
    By Construction \ref{p2cons T1}, the denominator of $\alpha_1$ is strictly positive, and the denominator of $\alpha_2$ is strictly negative. By Lemma \ref{cases lemma}, the numerator of both $\alpha_1$ and $\alpha_2$ is strictly positive, and the thesis follows.
\end{proof}

\begin{lem} \label{cases lemma 2}
    Let $X=X(g,P)$ be a non-toric Gorenstein Fano threefold of Type B arising form Construction \ref{p2cons T1}. With the notation of Proposition \ref{gorensteincor2}, we have
    \begin{itemize}
    \item $\nu-\lambda_1-\lambda_2-\sum_{i=3}^n\lambda_i=1+\sum_{i=1}^n d_i$;
    \item $\dfrac{1}{l_{01}}>1+\sum_{i=1}^n d_i-\dfrac{1}{l_{11}}-\dfrac{1}{l_{21}}-\sum_{i=3}^n\dfrac{d_i}{l_{i1}}>\dfrac{1}{l_{02}}.$
    \end{itemize}
    In particular, we have the following cases:
    \begin{enumerate}
        \item $n=3,\ d_3=1,\ l_{11}=1,\ l_{01}=1,\ 2\leq l_{21},\ 2\leq l_{31}$;
        \item $n=3,\ d_3=1,\ l_{11}=1,\ l_{01}=2,\ l_{21}=2,\ 2\leq l_{31}$;
        \item $n=3,\ d_3=1,\ l_{11}=1,\ l_{01}=2,\ l_{21}=3,\ 2\leq l_{31}\leq 5$;
        \item $n=3,\ d_3=1,\ l_{11}=1,\ 3\leq l_{01}\leq 5,\ l_{21}=2,\ 2\leq l_{31}\leq 5$;
        \item $n=3,\ d_3=1,\ l_{11}=2,\ l_{01}=1,\ l_{21}=2,\ 2\leq l_{31}$;
        \item $n=3,\ d_3=1,\ l_{11}=2,\ l_{01}=1,\ l_{21}=3,\ 2\leq l_{31}\leq 5$;
        \item $n=3,\ d_3=2,\ l_{11}=1,\ l_{01}=1,\ l_{21}=1,\ 3\leq l_{31}$;
        \item $n=3,\ d_3=2,\ l_{11}=1,\ l_{01}=2,\ l_{21}=1,\ l_{31}=3$;
        \item $n=3,\ d_3=2,\ l_{11}=1,\ l_{01}=1,\ 2\leq l_{21},\ l_{31}=2$;
        \item $n=3,\ d_3=2,\ l_{11}=1,\ l_{01}=1,\ l_{21}=2,\ l_{31}=3$;
        \item $n=3,\ d_3=3,\ l_{11}=1,\ l_{01}=1,\ l_{21}=1,\ l_{31}=2$;
        \item $n=4,\ d_3=d_4=1,\ l_{11}=1,\ l_{01}=1,\ l_{21}=1,\ 2\leq l_{31},\ 2\leq l_{41}$.
    \end{enumerate}
\end{lem}

\begin{proof}
    We have
    \begin{align*}
        &\dfrac{\mu d_{02}}{l_{02}}>\dfrac{\mu d_{11}}{l_{11}}+\dfrac{\mu d_{21}}{l_{21}}+\sum_{i=3}^n d_i\dfrac{\mu d_{i1}}{l_{i1}}>\dfrac{\mu d_{01}}{l_{01}}\implies
        \\
        &\nu-\dfrac{1}{l_{02}}>\lambda_1+\lambda_2+\sum_{i=3}^n\lambda_i+\dfrac{1}{l_{11}}+\dfrac{1}{l_{21}}+\sum_{i=3}^n\dfrac{d_i}{l_{i1}}>\nu-\dfrac{1}{l_{01}}\implies
        \\
        &1\geq\dfrac{1}{l_{01}}>\nu-\lambda_1-\lambda_2-\sum_{i=3}^n\lambda_i-\dfrac{1}{l_{11}}-\dfrac{1}{l_{21}}-\sum_{i=3}^n\dfrac{d_i}{l_{i1}}>\dfrac{1}{l_{02}}>0.
    \end{align*}
    By Lemma \ref{cases lemma} it follows
    $$
    1+\dfrac{1}{l_{11}}+\dfrac{1}{l_{21}}+\sum_{i=3}^n\dfrac{d_i}{2}>\dfrac{1}{l_{11}}+\dfrac{1}{l_{21}}+\sum_{i=3}^n\dfrac{d_i}{l_{i1}}>\sum_{i=3}^n d_i,
    $$
    and hence $\nu-\lambda_1-\lambda_2-\sum_{i=3}^n\lambda_i=1+\sum_{i=3}^n d_i.$
\end{proof}

\begin{lem} \label{unit fraction T2}
    Let $X=X(g,P)$ be a non-toric Gorenstein Fano threefold of Type B arising form Construction \ref{p2cons T1}. With the notation of Proposition \ref{gorensteincor2}, we have that
    {\footnotesize
        $$
        1+\sum_{i=3}^n d_i=\dfrac{1}{l_{02}}+\dfrac{1}{l_{11}}+\dfrac{1}{l_{21}}+\sum_{i=3}^n\dfrac{d_i}{l_{i1}}+l_{01}\dfrac{\dfrac{1}{l_{01}}+\dfrac{1}{l_{11}}+\dfrac{1}{l_{21}}+\sum_{i=3}^n\dfrac{d_i}{l_{i1}}-1-\sum_{i=3}^n d_i}{l_{01}\left(\gamma_1+\gamma_2+\sum_{i=3}^n d_i(\gamma_i-1)-1\right)+1+d_{01}\alpha_2}.
        $$
    }
\end{lem}

\begin{proof}
    By multiplying the numerator and denominator of $\alpha_2$ by $\mu$ and performing the same operations of Lemma \ref{cases lemma 2} we obtain
    \begin{align*}
        &\alpha_2=\mu\dfrac{\dfrac{1}{l_{02}}+\dfrac{1}{l_{11}}+\dfrac{1}{l_{21}}+\sum_{i=3}^n\dfrac{d_i}{l_{i1}}-\sum_{i=3}^n d_i}{\dfrac{1}{l_{02}}+\dfrac{1}{l_{11}}+\dfrac{1}{l_{21}}+\sum_{i=3}^n\dfrac{d_i}{l_{i1}}-1-\sum_{i=3}^n d_i}\implies\\
        & (\mu-\alpha_2)\left(\dfrac{1}{l_{02}}+\dfrac{1}{l_{11}}+\dfrac{1}{l_{21}}+\sum_{i=3}^n\dfrac{d_i}{l_{i1}}\right)=(\mu-\alpha_2)\left(\sum_{i=3}^n d_i+1\right)-\mu,
    \end{align*}
    and hence
    $$
    1+\sum_{i=3}^n d_i=\dfrac{1}{l_{02}}+\dfrac{1}{l_{11}}+\dfrac{1}{l_{21}}+\sum_{i=3}^n\dfrac{d_i}{l_{i1}}+\dfrac{\mu}{\mu-\alpha_2}.
    $$
    Furthermore, by Lemma \ref{cases lemma 2} we have
    \begin{align*}
        &\mu\left(\dfrac{d_{11}}{l_{11}}+\dfrac{d_{21}}{l_{21}}+\sum_{i=3}^n\dfrac{d_id_{i1}}{l_{i1}}-\dfrac{d_{01}}{l_{01}}\right)=\dfrac{1}{l_{01}}+\dfrac{1}{l_{11}}+\dfrac{1}{l_{21}}+\sum_{i=3}^n\dfrac{d_i}{l_{i1}}-1-\sum_{i=3}^n d_i\\
        & \implies \mu=\dfrac{\dfrac{1}{l_{01}}+\dfrac{1}{l_{11}}+\dfrac{1}{l_{21}}+\sum_{i=3}^n\dfrac{d_i}{l_{i1}}-1-\sum_{i=3}^n d_i}{\dfrac{d_{11}}{l_{11}}+\dfrac{d_{21}}{l_{21}}+\sum_{i=3}^n\dfrac{d_id_{i1}}{l_{i1}}-\dfrac{d_{01}}{l_{01}}}.
    \end{align*}
    Analogously,
    $$
    \alpha_2=\dfrac{\dfrac{1}{l_{01}}+\dfrac{1}{l_{11}}+\dfrac{1}{l_{21}}+\sum_{i=3}^n\dfrac{d_i}{l_{i1}}-\gamma_1-\gamma_2-\sum_{i=3}^n d_i\gamma_i-\dfrac{d_{01}\alpha_2-1}{l_{01}}}{\dfrac{d_{11}}{l_{11}}+\dfrac{d_{21}}{l_{21}}+\sum_{i=3}^n\dfrac{d_id_{i1}}{l_{i1}}-\dfrac{d_{01}}{l_{01}}}.
    $$
    By substituting both in $\dfrac{\mu}{\mu-\alpha_2}$ we get the thesis.
\end{proof}

\begin{proof}[Proof of Theorem \ref{big theorem}, Type B]
    First of all, notice that
    $$
    \mu=\dfrac{l_{02}-l_{01}}{l_{01}d_{02}-l_{02}d_{01}}=\dfrac{1}{\dfrac{d_{02}}{l_{02}}-\dfrac{d_{01}}{l_{01}}}\in\Z_{>0}\implies d_{02}\leq l_{02}\left(1+\dfrac{d_{01}}{l_{01}}\right)\leq 2l_{02}.
    $$
    Consider Case 1 of Lemma \ref{cases lemma 2}, that is $n=3,d_3=1,l_{01}=l_{11}=1.$ The condition of Lemma \ref{unit fraction T2} becomes
    $$
    1=\dfrac{1}{l_{02}}+\dfrac{1}{l_{21}}+\dfrac{1}{l_{31}}+\dfrac{\dfrac{1}{l_{21}}+\dfrac{1}{l_{31}}}{\gamma_2+\gamma_3+d_{01}\alpha_2}.
    $$
    Via Remark \ref{unit fraction remark}, we compute the triples of solutions $(l_{02},l_{21},l_{31})$. Since $d_{02}\leq 2l_{02}$, $d_{21}<l_{21}$ and $d_{31}<l_{31}$, we get finitely many possible values for the $l_{ij}$, $d_{ij}$, $d_3$. Finally, we check which combinations of those satisfy the conditions of Proposition \ref{fanoprop} and \ref{gorensteincor2}. The other cases of Lemma \ref{cases lemma 2} are handled analogously. This procedure identifies 69 families of non-toric Gorenstein Fano threefolds of Type B arising form Construction \ref{p2cons T1}.
\end{proof}

\begin{lem}
    There are 8 distinct families of non-toric Gorenstein Fano threefolds of Type C arising form Construction \ref{p2cons T1} with $l_{31}=l_{32}$.  
\end{lem}

\begin{proof}
    By Proposition \ref{gorensteincor2}, since $\mu=0$, we have $l_{i1}=1$ for every $i$. By Remark \ref{non-toricity}, it follows that $n=3$. Then
    $$\alpha_1=\dfrac{3}{-d_{01}}\in\Z_{<0},\hspace{5pt}\alpha_2=\dfrac{3}{d_3d_{32}-d_{01}}\in\Z_{>0},$$ hence $d_{01}=1,3$ and $d_3d_{32}-d_{01}=1,3$. Therefore, we get the following eight admissible triples $(d_{01},d_{32},d_3)$, each corresponding to a solution of the conditions in both Proposition \ref{fanoprop} and \ref{gorensteincor2}:
    \begin{align*}
        (1,1,2),\hspace{5pt}(1,2,2),\hspace{5pt}(3,2,2),\hspace{5pt}(3,3,2),\hspace{5pt}(3,2,3),\hspace{5pt}(1,1,4),\hspace{5pt}(3,1,4),\hspace{5pt}(3,1,6).
    \end{align*}
\end{proof}

\begin{lem} \label{condition 1 type C}
    Let $X=X(g,P)$ be a non-toric Gorenstein Fano threefold of Type C arising form Construction \ref{p2cons T1}. Then we have
    $$
    \dfrac{1}{l_{01}}+\dfrac{1}{l_{11}}+\dfrac{1}{l_{21}}+\dfrac{d_3}{l_{32}}+\sum_{i=4}^n\dfrac{d_i}{l_{i1}}>\sum_{i=3}^n d_i.
    $$
\end{lem}

\begin{proof}
    We use the notation of Proposition \ref{gorensteincor2}.
    Suppose that
    $$
    \dfrac{1}{l_{01}}+\dfrac{1}{l_{11}}+\dfrac{1}{l_{21}}+\sum_{i=4}^n\dfrac{d_i}{l_{i1}}-\sum_{i=3}^n d_i\leq-\dfrac{d_3}{l_{32}}.
    $$
    Then by Proposition \ref{fanoprop} we have
    \begin{align*}
    d_3\nu-\dfrac{d_3}{l_{32}}-\mu\left(\dfrac{d_{01}}{l_{01}}-\dfrac{d_{11}}{l_{11}}-\dfrac{d_{21}}{l_{21}}-\sum_{i=4}^n \dfrac{d_id_{i1}}{l_{i1}}\right)>0.
    \end{align*}
    Since $l_{32}\nu-1=d_{32}\mu$, it follows that
    $$
    \dfrac{d_3d_{32}}{l_{32}}-\dfrac{d_{01}}{l_{01}}+\dfrac{d_{11}}{l_{11}}+\dfrac{d_{21}}{l_{21}}+\sum_{i=4}^n \dfrac{d_id_{i1}}{l_{i1}}>0,
    $$
    contradicting Construction \ref{cons C}.
\end{proof}

\begin{lem} \label{unit fraction type c}
    Let $X=X(g,P)$ be a non-toric Gorenstein Fano threefold of Type C arising form Construction \ref{p2cons T1} with $l_{32}>l_{31}$. With the notation of Proposition \ref{gorensteincor2} we have
    \begin{align*}
    &\dfrac{1}{l_{01}}+\dfrac{1}{l_{11}}+\dfrac{1}{l_{21}}+\sum_{i=3}^n\dfrac{d_i}{l_{i1}}+\dfrac{\mu(d_3\nu-\lambda_0+\lambda_1+\lambda_2+\sum_{i=4}^n d_i\lambda_i-\sum_{i=3}^n d_i)}{\alpha_2+\mu}=\\
    &=d_3\nu-\lambda_0+\lambda_1+\lambda_2+\sum_{i=4}^n d_i\lambda_i.
    \end{align*}
    Furthermore, we have the following inequalities:
    \begin{itemize}
        \item $0<\nu<d_{32}<l_{32}$,
        \item $0\leq d_{31}<l_{31}<l_{32}$,
        \item $\lambda_i,\nu<\mu<l_{32}$,
        \item $\sum_{i=3}^n d_i\leq 5$,
        \item $\lambda_0-\lambda_1-\lambda_2-\sum_{i=4}^n d_i\lambda_i+\sum_{i=3}^n d_i<d_3\nu$.
    \end{itemize}
\end{lem}

\begin{proof}
    Since $l_{32}>l_{31}$, by Lemma \ref{condition 2 type C} we have
    $$
    3>\dfrac{1}{l_{01}}+\dfrac{1}{l_{11}}+\dfrac{1}{l_{21}}>\dfrac{\sum_{i=3}^n d_i}{2}\implies \sum_{i=3}^n d_i\leq 5.
    $$
    In particular, $n\leq 7$. Furthermore, since
    $$
    d_{01}=\dfrac{\lambda_0 l_{01}+1}{\mu},\hspace{5pt} d_{i1}=\dfrac{\lambda_1 l_{i1}-1}{\mu} \text{~for $i\geq 1$},\hspace{5pt} d_{32}=\dfrac{\nu l_{32}-1}{\mu},
    $$
    we have
    \begin{align*}
    &d_3 \dfrac{d_{32}\mu}{l_{32}}>\lambda_0-\lambda_1-\lambda_2-\sum_{i=4}^n d_i\lambda_i+\dfrac{1}{l_{01}}+\dfrac{1}{l_{11}}+\dfrac{1}{l_{21}}+\sum_{i=4}^n\dfrac{d_i}{l_{i1}}>\\
    &>\lambda_0-\lambda_1-\lambda_2-\sum_{i=4}^n\lambda_i+\sum_{i=3}^n d_i-\dfrac{d_3}{l_{32}}\implies\\
    & d_3\nu>\lambda_0-\lambda_1-\lambda_2-\sum_{i=4}^n d_i\lambda_i+\sum_{i=3}^n d_i.
    \end{align*}
    Furthermore, by rearranging $\alpha_2$ we get
    \begin{align*}
    &\alpha_2=\mu\dfrac{\dfrac{1}{l_{01}}+\dfrac{1}{l_{11}}+\dfrac{1}{l_{21}}+\dfrac{1}{l_{32}}+\sum_{i=4}^n\dfrac{d_i}{l_{i1}}-\sum_{i=3}^nd_i}{d_3\nu-\lambda_0+\lambda_1+\lambda_2+\sum_{i=4}^n d_i\lambda_i-\dfrac{1}{l_{01}}-\dfrac{1}{l_{11}}-\dfrac{1}{l_{21}}-\sum_{i=3}^n\dfrac{d_i}{l_{i1}}}\implies\\
    &\dfrac{1}{l_{01}}+\dfrac{1}{l_{11}}+\dfrac{1}{l_{21}}+\sum_{i=3}^n\dfrac{d_i}{l_{i1}}+\dfrac{\mu(d_3\nu-\lambda_0+\lambda_1+\lambda_2+\sum_{i=4}^n d_i\lambda_i-\sum_{i=3}^n d_i)}{\alpha_2+\mu}=\\
    &=d_3\nu-\lambda_0+\lambda_1+\lambda_2+\sum_{i=4}^n d_i\lambda_i.
    \end{align*}
    The other inequalities follow immediately by Remark \ref{non-redundancy} and Proposition \ref{gorensteincor2}.
\end{proof}

\begin{lem} \label{condition 2 type C}
    Let $X=X(g,P)$ be a non-toric Gorenstein Fano threefold of Type C arising form Construction \ref{p2cons T1} with $l_{32}>12$. Then $n=3$ and $d_3=2$. In particular, we have
    $$
    3\geq\dfrac{1}{l_{01}}+\dfrac{1}{l_{11}}+\dfrac{1}{l_{21}}>\dfrac{11}{6}.
    $$
\end{lem}

\begin{proof}
    We use the notation of Proposition \ref{gorensteincor2}.
    By Lemma \ref{condition 1 type C} we have
    $$
    3\geq\dfrac{1}{l_{01}}+\dfrac{1}{l_{11}}+\dfrac{1}{l_{21}}>\dfrac{11}{12}d_3+\sum_{i=4}^n \dfrac{d_i}{2}\implies n=3,4.
    $$
    Assume that $n=4$. Then $d_3=2$ and we have
    $$
    3\geq\dfrac{1}{l_{01}}+\dfrac{1}{l_{11}}+\dfrac{1}{l_{21}}>\dfrac{11}{6}+\dfrac{(l_{41}-1)d_4}{l_{41}}>2\implies l_{11}=l_{21}=1, d_4=l_{01}=l_{41}=2.
    $$
    Substituting these values in the condition of Construction \ref{cons C} yields
    $$
    2>2\dfrac{d_{32}}{l_{32}}>d_{01}-1>2\dfrac{d_{31}}{l_{31}}>0\implies d_{01}=2.
    $$
    It follows that
    \begin{align*}
    &\beta_j=\dfrac{1-\dfrac{2d_{3j}}{2d_{3j}-l_{3j}}}{l_{3j}}=\dfrac{-1}{2d_{3j}-l_{3j}}\implies l_{31}=2d_{31}+1,\hspace{5pt} l_{32}=2d_{32}-1\implies \\
    &\nu=\dfrac{d_{32}-d_{31}}{d_{32}+d_{31}}\in\Z_{>0}\implies d_{31}=0, l_{31}=1\implies \mu=2-\dfrac{2}{d_{32}}\in\Z_{>0},
    \end{align*}
    which in turn implies $d_{32}\leq2$ and $l_{32}=2d_{32}-1\leq 3$, contradicting $l_{32}>12$.
    It follows that $n=3$ and $d_3\leq 3$. Assume now $d_3=3$. Then $l_{01}=l_{11}=l_{21}=1$, and by Construction \ref{cons C} we have
    $$
    3>3\dfrac{d_{32}}{l_{32}}>d_{01}>3\dfrac{d_{31}}{l_{31}}\geq0\implies d_{01}=1,2.
    $$
    Furthermore,
    \begin{align*}
    &\alpha_j=\dfrac{3}{3d_{3j}-d_{01}l_{3j}},\hspace{5pt}\beta_j=\dfrac{-d_{01}}{3d_{3j}-d_{01}l_{3j}}\implies 3d_{3j}-d_{01}l_{3j}|1\\
    &\implies l_{31}=\dfrac{3d_{31}+1}{d_{01}},\hspace{5pt}l_{32}=\dfrac{3d_{32}-1}{d_{01}}\implies \nu=d_{01}\dfrac{d_{32}-d_{31}}{d_{32}+d_{31}}\in\Z_{>0}.
    \end{align*}
    If $d_{31}=0$, $l_{31}=1$ we have
    $$
    \mu=3-\dfrac{2}{d_{32}}\implies d_{32}=1,2\implies \dfrac{1}{2}>3\dfrac{d_{32}}{l_{32}}>d_{01}\geq 1,
    $$
    a contradiction.
    Hence $d_{31}\geq 1$. From $\nu\in\Z_{>0}$ it follows that $d_{01}=2$, and then
    $$
    d_{32}=3d_{31}\implies \mu=\dfrac{3}{2}-\dfrac{1}{2d_{31}}\implies d_{31}=1,d_{32}=3\implies \dfrac{3}{4}\geq\dfrac{9}{l_{32}}>1,
    $$
    a contradiction. We conclude that $d_3=2$.
\end{proof}

\begin{lem}\label{no solution above 12}
    There are no non-toric Gorenstein Fano threefolds of Type C arising form Construction \ref{p2cons T1} with $l_{32}>12$.
\end{lem}

\begin{proof}
    We use the notation of Proposition \ref{gorensteincor2}. By Lemma \ref{condition 2 type C}, we have $n=3$, $d_3=2$ and
    $$
    3\geq\dfrac{1}{l_{01}}+\dfrac{1}{l_{11}}+\dfrac{1}{l_{21}}>\dfrac{11}{6}.
    $$
    In particular $l_{11}=1$ and either $l_{21}=l_{01}=2$ or $l_{21}=1$.
    In the first case, by Construction \ref{cons C} we have
    $$
    \dfrac{d_{32}}{3}>4\dfrac{d_{32}}{l_{32}}>d_{01}-1>0\implies d_{32}>3.
    $$
    Furthermore,
    $$
    \alpha_2=\dfrac{4}{4d_{32}-d_{01}l_{32}+l_{32}}\implies \alpha_2|4,
    $$
    but $\dfrac{1-\alpha_2}{2}\in\Z$ so $\alpha_2=1$ and we have $\dfrac{1-d_{32}}{l_{32}}\in\Z$, contradicting $3<d_{32}<l_{32}$.
    In the second case, we have
    \begin{align*}
    &2\dfrac{d_{32}}{l_{32}}>\dfrac{d_{01}}{l_{01}}>2\dfrac{d_{31}}{l_{31}}\implies 2\nu-\dfrac{1}{l_{01}}-\dfrac{2}{l_{32}}>\lambda_0>2\nu-\dfrac{1}{l_{01}}-\dfrac{2}{l_{31}}\\
    &\implies \dfrac{1}{l_{01}}+\dfrac{2}{l_{32}}<2\nu-\lambda_0<\dfrac{1}{l_{01}}+\dfrac{2}{l_{31}}\implies l_{31}\leq 3.
    \end{align*}
    If $l_{31}=3$ then $l_{01}=2$ and $2\nu-\lambda_0=1$, hence we have
    $$
    \dfrac{-1+d_{01}\mu}{2}=\lambda_0=2\nu-1=\dfrac{2d_{31}\mu-1}{3}\implies(3d_{01}-4d_{31})\mu=1\implies \mu=1.
    $$
    Then, by Lemma \ref{unit fraction type c} we get
    $$
    \dfrac{2}{l_{32}}+\dfrac{1}{\alpha_2+1}=\dfrac{1}{2}.
    $$
    We compute the solutions of the equation via Remark \ref{unit fraction remark} and we see that none of them satisfy the conditions of both Proposition \ref{fanoprop} and \ref{gorensteincor2}.
    If $l_{31}=2$ then $2\nu-\lambda_0=1$ and we have
    $$
    \dfrac{\lambda_0+1}{2}=\nu=\dfrac{\mu+1}{2}\implies \lambda=\mu\implies \mu=\dfrac{d_{01}\mu-1}{l_{01}}\implies l_{01}=d_{01}-\dfrac{1}{\mu}\in\Z,
    $$
    hence $\mu=1$. Via Remark \ref{unit fraction remark} we compute the solutions of
    $$
    \dfrac{1}{l_{01}}+\dfrac{2}{l_{32}}+\dfrac{1}{\alpha_2+1}=1,
    $$
    and we see that none of them satisfy the conditions of both Proposition \ref{fanoprop} and \ref{gorensteincor2}.
    If $l_{31}=1$ then $\nu=1$, $\mu=\dfrac{l_{32}-1}{d_{32}}$, $\lambda_0=0,1$. If $\lambda_0=0$ then $d_{01}=\mu=1$. As above, by Remark \ref{unit fraction remark} we compute all solutions of
    $$
    \dfrac{1}{l_{01}}+\dfrac{2}{l_{32}}+\dfrac{2}{\alpha_2+1}=2
    $$
    and we see that none of them satisfy the conditions of both Proposition \ref{fanoprop} and \ref{gorensteincor2}.
    If $\lambda_0=1$ then $l_{01}=d_{01}-1$, and hence
    $$
    \alpha_1=\dfrac{2d_{01}\mu-1}{-d_{01}}=-2\mu+\dfrac{1}{d_{01}}\in\Z\implies d_{01}=1, l_{01}=\mu-1.
    $$
    It follows that
    \begin{align*}
    &\alpha_2+\mu\dfrac{1-d_{32}\alpha_2}{l_{32}}=\dfrac{\mu-1}{d_{32}(\mu-1)-d_{32}-1}=\dfrac{1}{d_{32}-\dfrac{d_{32}+1}{\mu-1}}\in\Z\\
    &\implies d_{32}-\dfrac{d_{32}+1}{\mu-1}=\dfrac{1}{\alpha_2+\mu\beta_2}\implies \mu\geq 3\implies d_{32}-\dfrac{d_{32}+1}{2}=\\
    &=\dfrac{d_{32}-1}{2}\leq\dfrac{1}{\alpha_2+\mu\beta_2}\implies d_{32}=1,2.
    \end{align*}
    If $d_{32}=1$ then
    $$
    1-\dfrac{2}{\mu-1}=\dfrac{1}{k}\implies \dfrac{k-1}{k}=\dfrac{2}{\mu-1}\implies \mu=4,5\implies (l_{32},l_{01})=(5,3),(6,4),
    $$
    If $d_{32}=2$ then
    $$
    2-\dfrac{3}{\mu-1}=\dfrac{1}{k}\implies\mu=3,4\implies(l_{32},l_{01})=(7,2),(9,3).
    $$
    By direct computation, we see that none of the four possibilities satisfy the conditions of both Proposition \ref{fanoprop} and \ref{gorensteincor2}.
    We conclude that there are no Gorenstein Fano threefolds of Type C arising form Construction \ref{p2cons T1} with $l_{32}>12$.
\end{proof}

\begin{proof}[Proof of Theorem \ref{big theorem}, Type C]
    By Lemma \ref{condition 1 type C}, there are 10 families of Gorenstein Fano threefolds of Type C arising form Construction \ref{p2cons T1} with $l_{31}=l_{32}$.
    By Lemma \ref{no solution above 12}, there are no Gorenstein Fano threefolds of Type C arising form Construction \ref{p2cons T1} with $l_{32}>12$.
    Assume now that $12\geq l_{32}>l_{31}$.
    By Lemma \ref{unit fraction type c}, we have
    \begin{itemize}
        \item $0<\nu<d_{32}<l_{32}\leq 12$,
        \item $0\leq d_{31}<l_{31}<l_{32}\leq 12$,
        \item $\lambda_i,\nu<\mu<l_{32}\leq 12$,
        \item $\sum_{i=3}^n d_i\leq 5$,
        \item $\lambda_0-\lambda_1-\lambda_2-\sum_{i=4}^n d_i\lambda_i+\sum_{i=3}^n d_i<d_3\nu\leq 60$.
    \end{itemize}
   \noindent
    For each possible fixed value of $ d_3\nu-\lambda_0+\lambda_1+\lambda_2+\sum_{i=4}^n d_i\lambda_i-\sum_{i=3}^n d_i$ we follow the procedure of Remark \ref{unit fraction remark} to compute the solutions of
    \begin{align*}
    &\dfrac{1}{l_{01}}+\dfrac{1}{l_{11}}+\dfrac{1}{l_{21}}+\sum_{i=3}^n\dfrac{d_i}{l_{i1}}+\dfrac{\mu(d_3\nu-\lambda_0+\lambda_1+\lambda_2+\sum_{i=4}^n d_i\lambda_i-\sum_{i=3}^n d_i)}{\alpha_2+\mu}=\\
    &=d_3\nu-\lambda_0+\lambda_1+\lambda_2+\sum_{i=4}^n d_i\lambda_i.
    \end{align*}
     Lastly, we check which solutions satisfy the conditions of Proposition \ref{fanoprop} and \ref{gorensteincor2}, identifying 25 families of Gorenstein Fano threefolds of Type C arising form Construction \ref{p2cons T1} with $12\geq l_{32}>l_{31}$.
\end{proof}

\begin{proof}[Proof of Theorem \ref{big theorem}, completeness and non-redundancy] By Proposition \ref{we get all}, every non-toric, $\Q$-factorial, projective threefold of Picard number one, invariant by the action of a one-dimensional torus and with a geometric maximal orbit quotient over $\Pp_2$ is $\K^*$-equivariantly isomorphic to an explicit $\T$-variety arising form Construction \ref{p2cons T1}. By Lemma \ref{unique moq}, if two such varieties are $\K^*$-equivariantly isomorphic then they share the same degree vector $d=(d_3,\dots,d_n)$, they are of the same type and they show the same isotropy orders among their points. Furthermore, they must share the same class group and Hilbert function. By checking the classification lists of Section \ref{lists section}, we see that no two families satisfy all of the above requirements. Hence, varieties of different families are not $\K^*$-equivariantly isomorphic.
\end{proof}

\section{Classification lists}\label{lists section}

 We describe the families determined in Section \ref{proof of theorem} in terms of Cox ring data. We write $f_i, f'_j, f''_k\in\C[T_0,T_1,T_2]$ to denote pairwise distinct, irreducible, homogeneous polynomials of degree $i,j,k$ respectively. Recall that the polynomials have to be chosen in such a way that the corresponding $X$ is an explicit $\T$-variety, see Proposition \ref{T-conditions prop 2}. For each family, we provide anticanonical degree and Hilbert function of its varieties.

\begin{lem}
    Let $X(g,P)$ be an explicit $\T$-variety arising from Construction \ref{p2cons T1}. With the notation of Definition \ref{fwv} and with
    $$-\mathcal{K}_X=(w_X,\eta_X)\in\Cl(X)=\Z\oplus\bigoplus_{i=1}^{q}\Z/a_i\Z,$$
    the anticanonical degree of $X$ is given according to type by
    \begin{itemize}
        \item[A:] \quad $-\mathcal{K}_X^3=w_X^3\dfrac{l_{31}\cdots l_{n1}}{a_1\cdots a_q w_{01}w_{11}w_{21}w_1},$
        \item[B:] \quad $-\mathcal{K}_X^3=w_X^3\dfrac{l_{31}\cdots l_{n1}}{a_1\cdots a_qw_{01}w_{02}w_{11}w_{21}},$
        \item[C:] \quad $-\mathcal{K}_X^3=w_X^3\dfrac{l_{31}l_{32}l_{41}\cdots l_{n1}}{a_1\cdots a_qw_{01}w_{11}w_{21}}.$
    \end{itemize}
\end{lem}

\begin{proof}
    We prove the lemma for Type A (B and C follow analogously).
    First, by \cite[Construction 3.3.3.4]{MR3307753} we have
    \begin{align*}
    \omega_{01}\cdot \omega_{11} \cdot \omega_{21}=(\omega_{01}\cdot \omega_{11}\cdot \omega_{21}\cdot \omega_{31}\cdots \omega_{n1})l_{31}\cdots l_{n1}=\dfrac{l_{31}\cdots l_{n1}}{a_1\cdots a_qw_1}.
    \end{align*}
    Furthermore, for every $\omega=(w,\eta)\in\Cl(X)$ we have
    $$
    a_1\cdots a_q \omega = a_1\cdots a_q w
^3
    =w_X^3\dfrac{\omega_{01}\cdot \omega_{11}\cdot \omega_{21}}{w_{01}w_{11}w_{21}}=w_X^3\dfrac{l_{31}\cdots l_{n1}}{a_1\cdots a_q w_{01}w_{11}w_{21}w_1}.
    \end{align*}
\end{proof}

\begin{classification}[$\Cl(X)=\Z$]
    Non-toric, $\Q$-factorial, Gorenstein Fano threefolds with a $\K^*$-action, Picard number one and maximal orbit quotient $\Pp_2$: specifying data for divisor class group $\Z$ and Type A.
    {\tiny
    \begin{center}
    \begin{longtable}{c| >{$}c<{$}|>{$}c<{$}|>{$}c<{$}|>{$}c<{$}|>{$}c<{$}}
        ID & \Rr(X) & Q & -\mathcal{K}_X & -\mathcal{K}_X^3 & HP_X(t)
        \\[2pt]
        \hline
        &&&&\\
        1
        &
        \dfrac{\K[T_{01},\dots,T_{31},T_1]}{\langle T_{11}^2+T_{21}^3+T_{31}^7\rangle}&
        \begin{bmatrix}
            42 & 21 & 14 & 6 & 1
        \end{bmatrix}
        &
        \begin{bmatrix}
         42
        \end{bmatrix}
        & 42
        &
        \dfrac{1+20t+20t^2+t^3}{1-t^4}
        \\[14 pt]
        \hline
        &&&&\\
        2&
        \dfrac{\K[T_{01},\dots,T_{31},T_1]}{\langle f_2(T_{01},T_{11},T_{21})+T_{31}^3\rangle}& 
        \begin{bmatrix}
            3 & 3 & 3 & 2 & 1
        \end{bmatrix}
        &
        \begin{bmatrix}
         6
        \end{bmatrix}
        &
        24
        &
        \dfrac{1+11t+11t^2+t^3}{1-t^4}
        \\[14 pt]
        \hline
        &&&&\\
        3&
        \dfrac{\K[T_{01},\dots,T_{31},T_1]}{\langle f_2(T_{01},T_{11},T_{21}^2)+T_{31}^5\rangle} & 
        \begin{bmatrix}
            10 & 10 & 5 & 4 & 1
        \end{bmatrix}
        &
        \begin{bmatrix}
         10
        \end{bmatrix}
        &
        10
        &
        \dfrac{1+4t+4t^2+t^3}{1-t^4}
        \\[14 pt]
        \hline
        &&&&\\
\end{longtable}
    \end{center}
    }
\end{classification}

\begin{classification}[$\Cl(X)=\Z$]
    Non-toric, $\Q$-factorial, Gorenstein Fano threefolds with a $\K^*$-action, Picard number one and maximal orbit quotient $\Pp_2$: specifying data for divisor class group $\Z$ and Type B.
    {\tiny
    \begin{center}
    \begin{longtable}{c| >{$}c<{$}|>{$}c<{$}|>{$}c<{$}|>{$}c<{$}|>{$}c<{$}}
    ID & \Rr(X) & Q & -\mathcal{K}_X & -\mathcal{K}_X^3 & HP_X(t)
        \\[2pt]
        \hline
        &&&&\\
    10 & \dfrac{\K[T_{01},\dots,T_{31}]}{\langle T_{01}T_{02}^{8}+T_{21}^{2}+T_{31}^{5}\rangle} &
    \begin{bmatrix}
    2 & 1 & 10 & 5 & 2
    \end{bmatrix} &
    \begin{bmatrix}
    10
    \end{bmatrix}
     & 50 & \dfrac{1+24t+24t^2+t^3}{1-t^4}
    \\[14 pt]
    \hline
    &&&&\\
    11 & \dfrac{\K[T_{01},\dots,T_{31}]}{\langle T_{01}T_{02}^{8}+T_{21}^{3}+T_{31}^{4}\rangle} &
    \begin{bmatrix}
    4 & 1 & 12 & 4 & 3
    \end{bmatrix} &
    \begin{bmatrix}
    12
    \end{bmatrix}
     & 36 & \dfrac{1+17t+17t^2+t^3}{1-t^4}
    \\[14 pt]
    \hline
    &&&&\\
    12 & \dfrac{\K[T_{01},\dots,T_{31}]}{\langle T_{01}T_{02}^{10}+T_{21}^{4}+T_{31}^{5}\rangle} &
    \begin{bmatrix}
    10 & 1 & 20 & 5 & 4
    \end{bmatrix} &
    \begin{bmatrix}
    20
    \end{bmatrix}
     & 40 & \dfrac{1+19t+19t^2+t^3}{1-t^4}
    \\[14 pt]
    \hline
    &&&&\\
\end{longtable}
    \end{center}
    }
\end{classification}

\begin{classification}[$\Cl(X)=\Z$]
    Non-toric, $\Q$-factorial, Gorenstein Fano threefolds with a $\K^*$-action, Picard number one and maximal orbit quotient $\Pp_2$: specifying data for divisor class group $\Z$ and Type C.
    {\tiny
    \begin{center}
    \begin{longtable}{c| >{$}c<{$}|>{$}c<{$}|>{$}c<{$}|>{$}c<{$}|>{$}c<{$}}
    ID & \Rr(X) & Q & -\mathcal{K}_X & -\mathcal{K}_X^3 & HP_X(t)
        \\[2pt]
        \hline
        &&&&\\
        40 & \dfrac{\K[T_{01},\dots,T_{31},T_{32}]}{\langle f_2(T_{01},T_{11},T_{21})+T_{31}T_{32}\rangle} &
    \begin{bmatrix}
    1 & 1 & 1 & 1 & 1
    \end{bmatrix} &
    \begin{bmatrix}
    3
    \end{bmatrix}
     & 54 & \dfrac{1+26t+26t^2+t^3}{1-t^4}
    \\[14 pt]
    \hline
    &&&&\\
    41 & \dfrac{\K[T_{01},\dots,T_{31},T_{32}]}{\langle f_2(T_{01}^{2},T_{11},T_{21})+T_{31}T_{32}^{2}\rangle} &
    \begin{bmatrix}
    1 & 2 & 2 & 2 & 1
    \end{bmatrix} &
    \begin{bmatrix}
    4
    \end{bmatrix}
     & 32 & \dfrac{1+15t+15t^2+t^3}{1-t^4}
    \\[14 pt]
    \hline
    &&&&\\
    42 & \dfrac{\K[T_{01},\dots,T_{31},T_{32}]}{\langle f_2(T_{01}^{3},T_{11},T_{21})+T_{31}T_{32}^{5}\rangle} &
    \begin{bmatrix}
    1 & 3 & 3 & 1 & 1
    \end{bmatrix} &
    \begin{bmatrix}
    3
    \end{bmatrix}
     & 18 & \dfrac{1+8t+8t^2+t^3}{1-t^4}
    \\[14 pt]
    \hline
    &&&&\\
\end{longtable}
    \end{center}
    }
\end{classification}

\begin{classification}[$\Cl(X)=\Z\oplus \Z/2\Z$]
    Non-toric, $\Q$-factorial, Gorenstein Fano threefolds with a $\K^*$-action, Picard number one and maximal orbit quotient $\Pp_2$: specifying data for divisor class group $\Z\oplus \Z/2\Z$ and Type A.
    {\tiny
    \begin{center}
    \begin{longtable}{c| >{$}c<{$}|>{$}c<{$}|>{$}c<{$}|>{$}c<{$}|>{$}c<{$}}
        ID & \Rr(X) & Q & -\mathcal{K}_X & -\mathcal{K}_X^3 & HP_X(t)
        \\[2pt]
        \hline
        &&&&\\
        68
        &
        \dfrac{\K[T_{01},\dots,T_{31},T_1]}{\langle T_{11}^2+T_{21}^3+T_{31}^8\rangle}
        &
        \begin{bmatrix}
            24 & 12 & 8 & 3 & 1 \\
            \bar{0} & \bar{1} & \bar{0} & \bar{1} & \bar{0}
        \end{bmatrix}
        &
         \begin{bmatrix}
         24 \\
         \bar{0}
        \end{bmatrix}
          & 24 & \dfrac{1+11t+11t^2+t^3}{1-t^4}
        \\[14 pt]
        \hline
        &&&&\\
        69 &
        \dfrac{\K[T_{01},\dots,T_{31},T_1]}{\langle T_{11}^2+T_{21}^4+T_{31}^5\rangle}&
        \begin{bmatrix}
            20 & 10 & 5 & 4 & 1 \\
            \bar{0} & \bar{1} & \bar{1} & \bar{0} & \bar{0}
        \end{bmatrix}
        &
        \begin{bmatrix}
         20 \\
         \bar{0}
        \end{bmatrix}
         & 20 & \dfrac{1+9t+9t^2+t^3}{1-t^4}
        \\[14 pt]
        \hline
        &&&&\\
        70 &
        \dfrac{\K[T_{01},\dots,T_{31},T_1]}{\langle f_2(T_{01},T_{11},T_{21})+T_{31}^2\rangle}&
        \begin{bmatrix}
            1 & 1 & 1 & 1 & 2 \\
            \bar{1} & \bar{1} & \bar{1} & \bar{0} & \bar{1}
        \end{bmatrix}
        &
        \begin{bmatrix}
         4 \\
         \bar{0}
        \end{bmatrix}
         & 32 & \dfrac{1+15t+15t^2+t^3}{1-t^4}
        \\[14 pt]
        \hline
        &&&&\\
\end{longtable}
    \end{center}
    }
\end{classification}

\begin{classification}[$\Cl(X)=\Z\oplus \Z/2\Z$]
    Non-toric, $\Q$-factorial, Gorenstein Fano threefolds with a $\K^*$-action, Picard number one and maximal orbit quotient $\Pp_2$: specifying data for divisor class group $\Z\oplus \Z/2\Z$ and Type B.
    {\tiny
    \begin{center}
    \begin{longtable}{c| >{$}c<{$}|>{$}c<{$}|>{$}c<{$}|>{$}c<{$}|>{$}c<{$}}
        ID & \Rr(X) & Q & -\mathcal{K}_X & -\mathcal{K}_X^3 & HP_X(t)
        \\[2pt]
        \hline
        &&&&\\
        81 & \dfrac{\K[T_{01},\dots,T_{31}]}{\langle T_{01}T_{02}^{6}+T_{21}^{2}+T_{31}^{8}\rangle} &
    \begin{bmatrix}
    2 & 1 & 8 & 4 & 1 \\
    \bar{0} & \bar{1} & \bar{0} & \bar{1} & \bar{0}
    \end{bmatrix} &
    \begin{bmatrix}
    8 \\
    \bar{0}
    \end{bmatrix}
     & 32 & \dfrac{1+15t+15t^2+t^3}{1-t^4}
    \\[14 pt]
    \hline
    &&&&\\
    82 & \dfrac{\K[T_{01},\dots,T_{31}]}{\langle T_{01}T_{02}^{6}+T_{21}^{4}+T_{31}^{6}\rangle} &
    \begin{bmatrix}
    6 & 1 & 12 & 3 & 2 \\
    \bar{0} & \bar{0} & \bar{0} & \bar{1} & \bar{1}
    \end{bmatrix} &
    \begin{bmatrix}
    12 \\
    \bar{0}
    \end{bmatrix}
     & 24 & \dfrac{1+11t+1t^2+t^3}{1-t^4}
    \\[14 pt]
    \hline
    &&&&\\
    83 & \dfrac{\K[T_{01},\dots,T_{31}]}{\langle T_{01}T_{02}^{7}+T_{21}^{2}+T_{31}^{4}\rangle} &
    \begin{bmatrix}
    1 & 1 & 8 & 4 & 2 \\
    \bar{0} & \bar{0} & \bar{0} & \bar{1} & \bar{1}
    \end{bmatrix} &
    \begin{bmatrix}
    8 \\
    \bar{0}
    \end{bmatrix}
     & 32 & \dfrac{1+15t+15t^2+t^3}{1-t^4}
    \\[14 pt]
    \hline
    &&&&\\
\end{longtable}
    \end{center}
    }
\end{classification}

\begin{classification}[$\Cl(X)=\Z\oplus \Z/2\Z$]
    Non-toric, $\Q$-factorial, Gorenstein Fano threefolds with a $\K^*$-action, Picard number one and maximal orbit quotient $\Pp_2$: specifying data for divisor class group $\Z\oplus \Z/2\Z$ and Type C.
    {\tiny
    \begin{center}
    \begin{longtable}{c| >{$}c<{$}|>{$}c<{$}|>{$}c<{$}|>{$}c<{$}|>{$}c<{$}}
        ID & \Rr(X) & Q & -\mathcal{K}_X & -\mathcal{K}_X^3 & HP_X(t)
        \\[2pt]
        \hline
        &&&&\\
        99 & \dfrac{\K[T_{01},\dots,T_{31},T_{32}]}{\langle f_2(T_{01}^{4},T_{11},T_{21}^{2})+T_{31}T_{32}^{4}\rangle} &
    \begin{bmatrix}
    1 & 4 & 2 & 4 & 1 \\
    \bar{1} & \bar{0} & \bar{1} & \bar{0} & \bar{0}
    \end{bmatrix} &
    \begin{bmatrix}
    4 \\
    \bar{0}
    \end{bmatrix}
     & 8 & \dfrac{1+3t+3t^2+t^3}{1-t^4}
    \\[14 pt]
    \hline
    &&&&\\
    100 & \dfrac{\K[T_{01},\dots,T_{31},T_{32}]}{\langle f_2(T_{01}^{4},T_{11},T_{21})+T_{31}^{2}T_{32}^{4}\rangle} &
    \begin{bmatrix}
    1 & 4 & 4 & 2 & 1 \\
    \bar{1} & \bar{0} & \bar{0} & \bar{1} & \bar{0}
    \end{bmatrix} &
    \begin{bmatrix}
    4 \\
    \bar{0}
    \end{bmatrix}
     & 8 & \dfrac{1+3t+3t^2+t^3}{1-t^4}
    \\[14 pt]
    \hline
    &&&&\\
    101 & \dfrac{\K[T_{01},\dots,T_{31},T_{32}]}{\langle f_2(T_{01}^{2},T_{11},T_{21})+T_{31}^{2}T_{32}^{6}\rangle} &
    \begin{bmatrix}
    2 & 4 & 4 & 1 & 1 \\
    \bar{1} & \bar{0} & \bar{0} & \bar{1} & \bar{0}
    \end{bmatrix} &
    \begin{bmatrix}
    4 \\
    \bar{0}
    \end{bmatrix}
     & 8 & \dfrac{1+3t+3t^2+t^3}{1-t^4}
     \\[14 pt]
    \hline
    &&&&\\
     102 & \dfrac{\K[T_{01},\dots,T_{31},T_{32},T_{41}]}{\left\langle\substack{f_2(T_{01}^{2},T_{11},T_{21})+T_{31}T_{32}^{2},\\f'_2(T_{01}^{2},T_{11},T_{21})+T_{41}^{2}}\right\rangle} &
    \begin{bmatrix}
    1 & 2 & 2 & 2 & 1 & 2 \\
    \bar{0} & \bar{0} & \bar{0} & \bar{0} & \bar{1} & \bar{1}
    \end{bmatrix} &
    \begin{bmatrix}
    2 \\
    \bar{0}
    \end{bmatrix}
     & 4 & \dfrac{1+t+t^2+t^3}{1-t^4}

        \end{longtable}
    \end{center}
    }
\end{classification}

\begin{classification}[$\Cl(X)=\Z\oplus \Z/3\Z$]
    Non-toric, $\Q$-factorial, Gorenstein Fano threefolds with a $\K^*$-action, Picard number one and maximal orbit quotient $\Pp_2$: specifying data for divisor class group $\Z\oplus \Z/3\Z$ and Type A.
    {\tiny
    \begin{center}
        \begin{longtable}{c| >{$}c<{$}|>{$}c<{$}|>{$}c<{$}|>{$}c<{$}|>{$}c<{$}}
        ID & \Rr(X) & Q & -\mathcal{K}_X & -\mathcal{K}_X^3 & HP_X(t)
        \\[2pt]
        \hline
        &&&&\\
        103 &
        \dfrac{\K[T_{01},\dots,T_{31},T_1]}{\langle T_{11}^2+T_{21}^3+T_{31}^9\rangle}&
        \begin{bmatrix}
            18 & 9 & 6 & 2 & 1 \\
            \bar{0} & \bar{0} & \bar{1} & \bar{2} & \bar{0}
        \end{bmatrix}
        &
        \begin{bmatrix}
         18 \\
         \bar{0}
        \end{bmatrix}
         & 18 & \dfrac{1+8t+8t^2+t^3}{1-t^4}
        \\[14pt]
        \hline
        &&&&\\
        104 &
        \dfrac{\K[T_{01},\dots,T_{31},T_1]}{\langle T_{11}^3+T_{21}^3+T_{31}^4\rangle}&
        \begin{bmatrix}
            12 & 4 & 4 & 3 & 1 \\
            \bar{0} & \bar{1} & \bar{2} & \bar{0} & \bar{0}
        \end{bmatrix}
        &
        \begin{bmatrix}
         12 \\
         \bar{0}
        \end{bmatrix}
        & 12 & \dfrac{1+5t+5t^2+t^3}{1-t^4}
        \\[14 pt]
        \hline
        &&&&\\
        105 &
        \dfrac{\K[T_{01},\dots,T_{31},T_1]}{\langle f_2(T_{01},T_{11},T_{21}^6)+T_{31}^3\rangle} & 
        \begin{bmatrix}
            6 & 6 & 1 & 4 & 1 \\
            \bar{0} & \bar{0} & \bar{2} & \bar{1} & \bar{0}
        \end{bmatrix}
        &
        \begin{bmatrix}
         6 \\
         \bar{0}
        \end{bmatrix}
         & 6 & \dfrac{1+2t+2t^2+t^3}{1-t^4}
        \end{longtable}
    \end{center}
    }
\end{classification}

\begin{classification}[$\Cl(X)=\Z\oplus \Z/3\Z$]
    Non-toric, $\Q$-factorial, Gorenstein Fano threefolds with a $\K^*$-action, Picard number one and maximal orbit quotient $\Pp_2$: specifying data for divisor class group $\Z\oplus \Z/3\Z$ and Type B.
    {\tiny
    \begin{center}
    \begin{longtable}{c| >{$}c<{$}|>{$}c<{$}|>{$}c<{$}|>{$}c<{$}|>{$}c<{$}}
        ID & \Rr(X) & Q & -\mathcal{K}_X & -\mathcal{K}_X^3 & HP_X(t)
        \\[2pt]
        \hline
        &&&&\\
        106 & \dfrac{\K[T_{01},\dots,T_{31}]}{\langle T_{01}T_{02}^{5}+T_{21}^{3}+T_{31}^{3}\rangle} &
    \begin{bmatrix}
    1 & 1 & 6 & 2 & 2 \\
    \bar{0} & \bar{0} & \bar{0} & \bar{2} & \bar{1}
    \end{bmatrix} &
    \begin{bmatrix}
    6 \\
    \bar{0}
    \end{bmatrix}
     & 18 & \dfrac{1+8t+8t^2+t^3}{1-t^4}
    \\[14 pt]
    \hline
    &&&&\\
    107 & \dfrac{\K[T_{01},\dots,T_{31}]}{\langle T_{01}T_{02}^{6}+T_{21}^{3}+T_{31}^{12}\rangle} &
    \begin{bmatrix}
    6 & 1 & 12 & 4 & 1 \\
    \bar{0} & \bar{2} & \bar{0} & \bar{1} & \bar{0}
    \end{bmatrix} &
    \begin{bmatrix}
    12 \\
    \bar{0}
    \end{bmatrix}
     & 24 & \dfrac{1+9t+9t^2+t^3}{1-t^4}
    \\[14 pt]
    \hline
    &&&&\\
    108 & \dfrac{\K[T_{01},\dots,T_{31}]}{\langle T_{01}T_{02}^{9}+T_{21}^{3}+T_{31}^{9}\rangle} &
    \begin{bmatrix}
    9 & 1 & 18 & 6 & 2 \\
    \bar{0} & \bar{0} & \bar{0} & \bar{1} & \bar{2}
    \end{bmatrix} &
    \begin{bmatrix}
    18 \\
    \bar{0}
    \end{bmatrix}
     & 18 & \dfrac{1+8t+8t^2+t^3}{1-t^4}
    \\[14 pt]
    \hline
    &&&&\\
\end{longtable}
    \end{center}
    }
\end{classification}

\begin{classification}[$\Cl(X)=\Z\oplus \Z/3\Z$]
    Non-toric, $\Q$-factorial, Gorenstein Fano threefolds with a $\K^*$-action, Picard number one and maximal orbit quotient $\Pp_2$: specifying data for divisor class group $\Z\oplus \Z/3\Z$ and Type C.
    {\tiny
    \begin{center}
        \begin{longtable}{c| >{$}c<{$}|>{$}c<{$}|>{$}c<{$}|>{$}c<{$}|>{$}c<{$}}
        ID & \Rr(X) & Q & -\mathcal{K}_X & -\mathcal{K}_X^3 & HP_X(t)
        \\[2pt]
        \hline
        &&&&\\
        114 & \dfrac{\K[T_{01},\dots,T_{31},T_{32}]}{\langle f_2(T_{01}^{3},T_{11},T_{21}^{3})+T_{31}T_{32}^{3}\rangle} &
    \begin{bmatrix}
    1 & 3 & 1 & 3 & 1 \\
    \bar{1} & \bar{0} & \bar{2} & \bar{0} & \bar{0}
    \end{bmatrix} &
    \begin{bmatrix}
    3 \\
    \bar{0}
    \end{bmatrix}
     & 6 & \dfrac{1+2t+2t^2+t^3}{1-t^4}
    \\[14 pt]
    \hline
    &&&&\\
    115 & \dfrac{\K[T_{01},\dots,T_{31},T_{32}]}{\langle f_2(T_{01},T_{11},T_{21})+T_{31}T_{32}\rangle} &
    \begin{bmatrix}
    1 & 1 & 1 & 1 & 1 \\
    \bar{2} & \bar{2} & \bar{2} & \bar{1} & \bar{0}
    \end{bmatrix} &
    \begin{bmatrix}
    3 \\
    \bar{0}
    \end{bmatrix}
     & 18 & \dfrac{1+5t+5t^2+t^3}{1-t^4}

        \end{longtable}
    \end{center}
    }
\end{classification}

\begin{classification}[$\Cl(X)=\Z\oplus \Z/4\Z$]
    Non-toric, $\Q$-factorial, Gorenstein Fano threefolds with a $\K^*$-action, Picard number one and maximal orbit quotient $\Pp_2$: specifying data for divisor class group $\Z\oplus \Z/4\Z$ and Type A.
    {\tiny
    \begin{center}
        \begin{longtable}{c| >{$}c<{$}|>{$}c<{$}|>{$}c<{$}|>{$}c<{$}|>{$}c<{$}}
        ID & \Rr(X) & Q & -\mathcal{K}_X & -\mathcal{K}_X^3 & HP_X(t)
        \\[2pt]
        \hline
        &&&&\\
        116 &
        \dfrac{\K[T_{01},\dots,T_{31},T_1]}{\langle f_2(T_{01},T_{11},T_{21}^4)+T_{31}^4\rangle} & 
        \begin{bmatrix}
            4 & 4 & 1 & 2 & 1 \\
            \bar{0} & \bar{0} & \bar{3} & \bar{1} & \bar{0}
        \end{bmatrix}
        &
        \begin{bmatrix}
         4 \\
         \bar{0}
        \end{bmatrix}
         & 4 & \dfrac{1+t+t^2+t^3}{1-t^4}
        \end{longtable}
    \end{center}
    }
\end{classification}
\begin{classification}[$\Cl(X)=\Z\oplus \Z/4\Z$]
    Non-toric, $\Q$-factorial, Gorenstein Fano threefolds with a $\K^*$-action, Picard number one and maximal orbit quotient $\Pp_2$: specifying data for divisor class group $\Z\oplus \Z/4\Z$ and Type B.
    {\tiny
    \begin{center}
    \begin{longtable}{c| >{$}c<{$}|>{$}c<{$}|>{$}c<{$}|>{$}c<{$}|>{$}c<{$}}
        ID & \Rr(X) & Q & -\mathcal{K}_X & -\mathcal{K}_X^3 & HP_X(t)
        \\[2pt]
        \hline
        &&&&\\
        117 & \dfrac{\K[T_{01},\dots,T_{31}]}{\langle T_{01}T_{02}^{4}+T_{21}^{4}+T_{31}^{8}\rangle} &
    \begin{bmatrix}
    4 & 1 & 8 & 2 & 1 \\
    \bar{0} & \bar{3} & \bar{0} & \bar{1} & \bar{0}
    \end{bmatrix} &
    \begin{bmatrix}
    8 \\
    \bar{0}
    \end{bmatrix}
     & 16 & \dfrac{1+9t+9t^2+t^3}{1-t^4}
        \end{longtable}
    \end{center}
    }
\end{classification}

\begin{classification}[$\Cl(X)=\Z\oplus \Z/5\Z$]
    Non-toric, $\Q$-factorial, Gorenstein Fano threefolds with a $\K^*$-action, Picard number one and maximal orbit quotient $\Pp_2$: specifying data for divisor class group $\Z\oplus \Z/5\Z$ and Type A.
    {\tiny
    \begin{center}
        \begin{longtable}{c| >{$}c<{$}|>{$}c<{$}|>{$}c<{$}|>{$}c<{$}|>{$}c<{$}}
        ID & \Rr(X) & Q & -\mathcal{K}_X & -\mathcal{K}_X^3 & HP_X(t)
        \\[2pt]
        \hline
        &&&&\\
        118 &
        \dfrac{\K[T_{01},\dots,T_{31},T_1]}{\langle T_{11}^2+T_{21}^5+T_{31}^5\rangle} &
        \begin{bmatrix}
            10 & 5 & 2 & 2 & 1 \\
            \bar{0} & \bar{0} & \bar{3} & \bar{2} & \bar{0}
        \end{bmatrix}
        &
        \begin{bmatrix}
         10 \\
         \bar{0}
        \end{bmatrix}
        &
         10 & \dfrac{1+4t+4t^2+t^3}{1-t^4}
        \end{longtable}
    \end{center}
    }
\end{classification}

\begin{classification}[$\Cl(X)=\Z\oplus \Z/5\Z$]
    Non-toric, $\Q$-factorial, Gorenstein Fano threefolds with a $\K^*$-action, Picard number one and maximal orbit quotient $\Pp_2$: specifying data for divisor class group $\Z\oplus \Z/5\Z$ and Type B.
    {\tiny
    \begin{center}
    \begin{longtable}{c| >{$}c<{$}|>{$}c<{$}|>{$}c<{$}|>{$}c<{$}|>{$}c<{$}}
        ID & \Rr(X) & Q & -\mathcal{K}_X & -\mathcal{K}_X^3 & HP_X(t)
        \\[2pt]
        \hline
        &&&&\\
        119 & \dfrac{\K[T_{01},\dots,T_{31}]}{\langle T_{01}T_{02}^{5}+T_{21}^{5}+T_{31}^{5}\rangle} &
    \begin{bmatrix}
    5 & 1 & 10 & 2 & 2 \\
    \bar{0} & \bar{0} & \bar{0} & \bar{3} & \bar{2}
    \end{bmatrix} &
    \begin{bmatrix}
    10 \\
    \bar{0}
    \end{bmatrix}
     & 10 & \dfrac{1+4t+4t^2+t^3}{1-t^4}

        \end{longtable}
    \end{center}
    }
\end{classification}

\begin{classification}[$\Cl(X)=\Z\oplus \Z/6\Z$]
    Non-toric, $\Q$-factorial, Gorenstein Fano threefolds with a $\K^*$-action, Picard number one and maximal orbit quotient $\Pp_2$: specifying data for divisor class group $\Z\oplus \Z/6\Z$ and Type A.
    {\tiny
    \begin{center}
        \begin{longtable}{c| >{$}c<{$}|>{$}c<{$}|>{$}c<{$}|>{$}c<{$}|>{$}c<{$}}
        ID & \Rr(X) & Q & -\mathcal{K}_X & -\mathcal{K}_X^3 & HP_X(t)
        \\[2pt]
        \hline
        &&&&\\
        120 &
        \dfrac{\K[T_{01},\dots,T_{31},T_1]}{\langle T_{11}^2+T_{21}^3+T_{31}^{12}\rangle} &
        \begin{bmatrix}
            12 & 6 & 4 & 1 & 1 \\
            \bar{0} & \bar{3} & \bar{4} & \bar{5} & \bar{0}
        \end{bmatrix}
        &
        \begin{bmatrix}
         12 \\
         \bar{0}
        \end{bmatrix}
         & 12 & \dfrac{1+5t+5t^2+t^3}{1-t^4}
        \end{longtable}
    \end{center}
    }
\end{classification}

\begin{classification}[$\Cl(X)=\Z\oplus (\Z/2\Z)^2$]
    Non-toric, $\Q$-factorial, Gorenstein Fano threefolds with a $\K^*$-action, Picard number one and maximal orbit quotient $\Pp_2$: specifying data for divisor class group $\Z\oplus (\Z/2\Z)^2$ and Type A.
    {\tiny
    \begin{center}
        \begin{longtable}{c| >{$}c<{$}|>{$}c<{$}|>{$}c<{$}|>{$}c<{$}|>{$}c<{$}}
        ID & \Rr(X) & Q & -\mathcal{K}_X & -\mathcal{K}_X^3 & HP_X(t)
        \\[2pt]
        \hline
        &&&&\\
        121
        &
        \dfrac{\K[T_{01},\dots,T_{31},T_1]}{\langle T_{11}^2+T_{21}^2+T_{31}^2\rangle} &
        \begin{bmatrix}
            2 & 1 & 1 & 1 & 1 \\
            \bar{0} & \bar{1} & \bar{1} & \bar{0} & \bar{0} \\
            \bar{0} & \bar{1} & \bar{0} & \bar{1} & \bar{0}
        \end{bmatrix}
        &
        \begin{bmatrix}
         4 \\
         \bar{0}\\
         \bar{0}
        \end{bmatrix}
         & 16 & \dfrac{1+7t+7t^2+t^3}{1-t^4}
        \\[14 pt]
        \hline
        &&&&\\
        122
        &
        \dfrac{\K[T_{01},\dots,T_{31},T_1]}{\langle T_{11}^2+T_{21}^2+T_{31}^2\rangle} &
        \begin{bmatrix}
            2 & 1 & 1 & 1 & 3 \\
            \bar{0} & \bar{1} & \bar{1} & \bar{0} & \bar{0} \\
            \bar{0} & \bar{0} & \bar{1} & \bar{0} & \bar{1}
        \end{bmatrix}
        &
        \begin{bmatrix}
         6 \\
         \bar{0}\\
         \bar{0}
        \end{bmatrix}
        & 18 & \dfrac{1+8t+8t^2+t^3}{1-t^4}
        \\[14 pt]
        \hline
        &&&&\\
        123 &
        \dfrac{\K[T_{01},\dots,T_{31},T_1]}{\langle T_{11}^2+T_{21}^4+T_{31}^6\rangle} &
        \begin{bmatrix}
            12 & 6 & 3 & 2 & 1 \\
            \bar{0} & \bar{0} & \bar{1} & \bar{1} & \bar{0} \\
            \bar{0} & \bar{1} & \bar{0} & \bar{1} & \bar{0}
        \end{bmatrix}
        &
        \begin{bmatrix}
         12 \\
         \bar{0}\\
         \bar{0}
        \end{bmatrix}
         & 12 & \dfrac{1+5t+5t^2+t^3}{1-t^4}
        \\[14 pt]
        \hline
        &&&&\\
\end{longtable}
    \end{center}
    }
\end{classification}

\begin{classification}[$\Cl(X)=\Z\oplus (\Z/2\Z)^2$]
    Non-toric, $\Q$-factorial, Gorenstein Fano threefolds with a $\K^*$-action, Picard number one and maximal orbit quotient $\Pp_2$: specifying data for divisor class group $\Z\oplus \Z/(2\Z)^2$ and Type B.
    {\tiny
    \begin{center}
    \begin{longtable}{c| >{$}c<{$}|>{$}c<{$}|>{$}c<{$}|>{$}c<{$}|>{$}c<{$}}
        ID & \Rr(X) & Q & -\mathcal{K}_X & -\mathcal{K}_X^3 & HP_X(t)
        \\[2pt]
        \hline
        &&&&\\
        130 & \dfrac{\K[T_{01},\dots,T_{31}]}{\langle T_{01}^{2}T_{02}^{4}+T_{21}^{2}+T_{31}^{6}\rangle} &
    \begin{bmatrix}
    1 & 1 & 6 & 3 & 1 \\
    \bar{1} & \bar{0} & \bar{0} & \bar{1} & \bar{0} \\
    \bar{0} & \bar{1} & \bar{0} & \bar{1} & \bar{0}
    \end{bmatrix} &
    \begin{bmatrix}
    6 \\
    \bar{0} \\
    \bar{0}
    \end{bmatrix}
     & 18 & \dfrac{1+8t+8t^2+t^3}{1-t^4}
    \\[14 pt]
    \hline
    &&&&\\
    131 & \dfrac{\K[T_{01},\dots,T_{31}]}{\langle T_{01}^{2}T_{02}^{4}+T_{21}^{2}+T_{31}^{8}\rangle} &
    \begin{bmatrix}
    2 & 1 & 8 & 4 & 1 \\
    \bar{1} & \bar{0} & \bar{0} & \bar{1} & \bar{0} \\
    \bar{0} & \bar{1} & \bar{0} & \bar{1} & \bar{0}
    \end{bmatrix} &
    \begin{bmatrix}
    8 \\
    \bar{0}\\
    \bar{2}
    \end{bmatrix}
     & 16 & \dfrac{1+7t+7t^2+t^3}{1-t^4}
    \\[14 pt]
    \hline
    &&&&\\
    132 & \dfrac{\K[T_{01},\dots,T_{31}]}{\langle T_{01}^{2}T_{02}^{6}+T_{21}^{2}+T_{31}^{4}\rangle} &
    \begin{bmatrix}
    1 & 1 & 8 & 4 & 2 \\
    \bar{1} & \bar{0} & \bar{0} & \bar{1} & \bar{0} \\
    \bar{0} & \bar{0} & \bar{0} & \bar{1} & \bar{1}
    \end{bmatrix} &
    \begin{bmatrix}
    8 \\
    \bar{0} \\
    \bar{0}
    \end{bmatrix}
     & 16 & \dfrac{1+7t+7t^2+t^3}{1-t^4}
    \\[14 pt]
    \hline
    &&&&\\
\end{longtable}
    \end{center}
    }
\end{classification}

\begin{classification}[$\Cl(X)=\Z\oplus (\Z/2\Z)^2$]
    Non-toric, $\Q$-factorial, Gorenstein Fano threefolds with a $\K^*$-action, Picard number one and maximal orbit quotient $\Pp_2$: specifying data for divisor class group $\Z\oplus \Z/(2\Z)^2$ and Type C.
    {\tiny
    \begin{center}
    \begin{longtable}{c| >{$}c<{$}|>{$}c<{$}|>{$}c<{$}|>{$}c<{$}|>{$}c<{$}}
        ID & \Rr(X) & Q & -\mathcal{K}_X & -\mathcal{K}_X^3 & HP_X(t)
        \\[2pt]
        \hline
        &&&&\\
    141 & \dfrac{\K[T_{01},\dots,T_{31},T_{32}]}{\langle f_2(T_{01}^{2},T_{11}^{2},T_{21}^{2})+T_{31}T_{32}^{2}\rangle} &
    \begin{bmatrix}
    1 & 1 & 1 & 2 & 1 \\
    \bar{1} & \bar{1} & \bar{0} & \bar{0} & \bar{0} \\
    \bar{0} & \bar{1} & \bar{1} & \bar{0} & \bar{0}
    \end{bmatrix} &
    \begin{bmatrix}
    2 \\
    \bar{0} \\
    \bar{0}
    \end{bmatrix}
     & 4 & \dfrac{1+t+t^2+t^3}{1-t^4}
    \\[14 pt]
    \hline
    &&&&\\
    142 & \dfrac{\K[T_{01},\dots,T_{31},T_{32},T_{41}]}{\left\langle \substack{f_2\left(T_{01}^2,T_{11},T_{21}^2\right)+T_{31}T_{32}^2,\\ f'_1\left(T_{01}^2,T_{11},T_{21}^2\right)+T_{41}^2}\right\rangle} &
    \begin{bmatrix}
    1 & 2 & 1 & 2 & 1 & 1 \\
    \bar{1} & \bar{0} & \bar{1} & \bar{0} & \bar{0} & \bar{0} \\
    \bar{0} & \bar{0} & \bar{1} & \bar{0} & \bar{1} & \bar{0}
    \end{bmatrix} &
    \begin{bmatrix}
    2 \\
    \bar{0} \\
    \bar{0}
    \end{bmatrix}
     & 4 & \dfrac{1+t+t^2+t^3}{1-t^4}

    \end{longtable}
    \end{center}
    }
\end{classification}

\begin{classification}[$\Cl(X)=\Z\oplus\Z/2\Z\oplus \Z/4Z$]
    Non-toric, $\Q$-factorial, Gorenstein Fano threefolds with a $\K^*$-action, Picard number one and maximal orbit quotient $\Pp_2$: specifying data for divisor class group $\Z\oplus\Z/2\Z\oplus \Z/4Z$ and Type A.
    {\tiny
    \begin{center}
        \begin{longtable}{c| >{$}c<{$}|>{$}c<{$}|>{$}c<{$}|>{$}c<{$}|>{$}c<{$}}
        ID & \Rr(X) & Q & -\mathcal{K}_X & -\mathcal{K}_X^3 & HP_X(t)
        \\[2pt]
        \hline
        &&&&\\
        143 &
        \dfrac{\K[T_{01},\dots,T_{31},T_1]}{\langle T_{11}^2+T_{21}^4+T_{31}^8\rangle}&
        \begin{bmatrix}
            8 & 4 & 2 & 1 & 1 \\
            \bar{0} & \bar{1} & \bar{0} & \bar{1} & \bar{0} \\
            \bar{0} & \bar{2} & \bar{3} & \bar{3} & \bar{0}
        \end{bmatrix}
        &
        \begin{bmatrix}
         8 \\
         \bar{0}\\
         \bar{0}
        \end{bmatrix}
         & 8 & \dfrac{1+3t+3t^2+t^3}{1-t^4}
            \end{longtable}
    \end{center}
    }
\end{classification}

\begin{classification}[$\Cl(X)=\Z\oplus\Z/2\Z\oplus \Z/6Z$]
    Non-toric, $\Q$-factorial, Gorenstein Fano threefolds with a $\K^*$-action, Picard number one and maximal orbit quotient $\Pp_2$: specifying data for divisor class group $\Z\oplus\Z/2\Z\oplus \Z/6Z$ and Type A.
    {\tiny
    \begin{center}
        \begin{longtable}{c| >{$}c<{$}|>{$}c<{$}|>{$}c<{$}|>{$}c<{$}|>{$}c<{$}}
        ID & \Rr(X) & Q & -\mathcal{K}_X & -\mathcal{K}_X^3 & HP_X(t)
        \\[2pt]
        \hline
        &&&&\\
        144 &
        \dfrac{\K[T_{01},\dots,T_{31},T_1]}{\langle T_{11}^2+T_{21}^6+T_{31}^6\rangle}&
        \begin{bmatrix}
            6 & 3 & 1 & 1 & 1 \\
            \bar{0} & \bar{0} & \bar{1} & \bar{1} & \bar{0} \\
            \bar{0} & \bar{3} & \bar{5} & \bar{4} & \bar{0}
        \end{bmatrix}
        &
        \begin{bmatrix}
         6 \\
         \bar{0}\\
         \bar{0}
        \end{bmatrix}
        & 6 & \dfrac{1+2t+2t^2+t^3}{1-t^4}
            \end{longtable}
    \end{center}
    }
\end{classification}

\begin{classification}[$\Cl(X)=\Z\oplus(\Z/3\Z)^2$]
    Non-toric, $\Q$-factorial, Gorenstein Fano threefolds with a $\K^*$-action, Picard number one and maximal orbit quotient $\Pp_2$: specifying data for divisor class group $\Z\oplus(\Z/3\Z)^2$ and Type A.
    {\tiny
    \begin{center}
        \begin{longtable}{c| >{$}c<{$}|>{$}c<{$}|>{$}c<{$}|>{$}c<{$}|>{$}c<{$}}
        ID & \Rr(X) & Q & -\mathcal{K}_X & -\mathcal{K}_X^3 & HP_X(t)
        \\[2pt]
        \hline
        &&&&\\
        145 &
        \dfrac{\K[T_{01},\dots,T_{31},T_1]}{\langle T_{11}^3+T_{21}^3+T_{31}^6\rangle} &
        \begin{bmatrix}
            6 & 2 & 2 & 1 & 1 \\
            \bar{0} & \bar{2} & \bar{2} & \bar{2} & \bar{0} \\
            \bar{0} & \bar{0} & \bar{1} & \bar{2} & \bar{0}
        \end{bmatrix}
        &
        \begin{bmatrix}
         6 \\
         \bar{0}\\
         \bar{0}
        \end{bmatrix}
         & 6 & \dfrac{1+2t+2t^2+t^3}{1-t^4}
            \end{longtable}
    \end{center}
    }
\end{classification}

\begin{classification}[$\Cl(X)=\Z\oplus(\Z/4\Z)^2$]
    Non-toric, $\Q$-factorial, Gorenstein Fano threefolds with a $\K^*$-action, Picard number one and maximal orbit quotient $\Pp_2$: specifying data for divisor class group $\Z\oplus(\Z/4\Z)^2$ and Type A.
    {\tiny
    \begin{center}
        \begin{longtable}{c| >{$}c<{$}|>{$}c<{$}|>{$}c<{$}|>{$}c<{$}|>{$}c<{$}}
        ID & \Rr(X) & Q & -\mathcal{K}_X & -\mathcal{K}_X^3 & HP_X(t)
        \\[2pt]
        \hline
        &&&&\\
        146 &
        \dfrac{\K[T_{01},\dots,T_{31},T_1]}{\langle T_{11}^4+T_{21}^4+T_{31}^4\rangle}&
        \begin{bmatrix}
            4 & 1 & 1 & 1 & 1 \\
            \bar{0} & \bar{3} & \bar{3} & \bar{2} & \bar{0} \\
            \bar{0} & \bar{3} & \bar{0} & \bar{1} & \bar{0}
        \end{bmatrix}
        &
        \begin{bmatrix}
         4 \\
         \bar{0}\\
         \bar{0}
        \end{bmatrix}
          & 4 & \dfrac{1+t+t^2+t^3}{1-t^4}
            \end{longtable}
    \end{center}
    }
\end{classification}

\begin{classification}[$\Cl(X)=\Z\oplus(\Z/2\Z)^3$]
    Non-toric, $\Q$-factorial, Gorenstein Fano threefolds with a $\K^*$-action, Picard number one and maximal orbit quotient $\Pp_2$: specifying data for divisor class group $\Z\oplus(\Z/2\Z)^3$ and Type A.
    {\tiny
    \begin{center}
        \begin{longtable}{c| >{$}c<{$}|>{$}c<{$}|>{$}c<{$}|>{$}c<{$}|>{$}c<{$}}
        ID & \Rr(X) & Q & -\mathcal{K}_X & -\mathcal{K}_X^3 & HP_X(t)
        \\[2pt]
        \hline
        &&&&\\
        147 &
        \dfrac{\K[T_{01},\dots,T_{31},T_1]}{\langle f_1(T_{01}^2,T_{11}^2,T_{21}^2)+T_{31}^2\rangle}&
        \begin{bmatrix}
            1 & 1 & 1 & 1 & 2 \\
            \bar{1} & \bar{1} & \bar{1} & \bar{0} & \bar{1} \\
            \bar{0} & \bar{0} & \bar{1} & \bar{0} & \bar{1} \\
            \bar{1} & \bar{0} & \bar{0} & \bar{0} & \bar{1}
        \end{bmatrix}
        &
        \begin{bmatrix}
         4 \\
         \bar{0}\\
         \bar{0}\\
         \bar{0}
        \end{bmatrix}
         & 8 & \dfrac{1+3t+3t^2+t^3}{1-t^4}
        \\[14 pt]
        \hline
        &&&&\\
        148 &
        \dfrac{\K[T_{01},\dots,T_{31},T_1]}{\langle f_1(T_{01}^2,T_{11}^2,T_{21}^2)+T_{31}^4\rangle}&
        \begin{bmatrix}
            2 & 2 & 2 & 1 & 1 \\
            \bar{1} & \bar{1} & \bar{1} & \bar{1} & \bar{0} \\
            \bar{0} & \bar{0} & \bar{1} & \bar{1} & \bar{0} \\
            \bar{1} & \bar{0} & \bar{0} & \bar{1} & \bar{0}
        \end{bmatrix}
        &
        \begin{bmatrix}
         4 \\
         \bar{0}\\
         \bar{0}\\
         \bar{0}
        \end{bmatrix}
         & 4 & \dfrac{1+t+t^2+t^3}{1-t^4}
        \\[14 pt]
        \hline
        &&&&\\
        149 &
        \dfrac{\K[T_{01},\dots,T_{31},T_1]}{\langle f_2(T_{01}^2,T_{11}^2,T_{21}^2)+T_{31}^2\rangle}& 
        \begin{bmatrix}
            1 & 1 & 1 & 2 & 1 \\
            \bar{1} & \bar{0} & \bar{0} & \bar{1} & \bar{0} \\
            \bar{0} & \bar{1} & \bar{0} & \bar{1} & \bar{0} \\
            \bar{0} & \bar{0} & \bar{1} & \bar{1} & \bar{0}
        \end{bmatrix}
        &
        \begin{bmatrix}
         2 \\
         \bar{0}\\
         \bar{0}\\
         \bar{0}
        \end{bmatrix}
         & 2 & \dfrac{1+t^3}{1-t^4}
        \\[14 pt]
        \hline
        &&&&\\
\end{longtable}
    \end{center}
    }
\end{classification}

\begin{classification}[$\Cl(X)=\Z\oplus(\Z/2\Z)^4$]
    Non-toric, $\Q$-factorial, Gorenstein Fano threefolds with a $\K^*$-action, Picard number one and maximal orbit quotient $\Pp_2$: specifying data for divisor class group $\Z\oplus(\Z/2\Z)^4$ and Type A.
    {\tiny
    \begin{center}
        \begin{longtable}{c| >{$}c<{$}|>{$}c<{$}|>{$}c<{$}|>{$}c<{$}|>{$}c<{$}}
        ID & \Rr(X) & Q & -\mathcal{K}_X & -\mathcal{K}_X^3 & HP_X(t)
        \\[2pt]
        \hline
        &&&&\\
        153 &
        \dfrac{\K[T_{01},\dots,T_{41},T_1]}{\left\langle \substack{f_1\left(T_{01}^2,T_{11}^2,T_{21}^2\right)+T_{31}^2,\\ f'_1\left(T_{01}^2,T_{11}^2,T_{21}^2\right)+T_{41}^2}\right\rangle}&
        \begin{bmatrix}
            1 & 1 & 1 & 1 & 1 & 1\\
            \bar{1} & \bar{1} & \bar{1} & \bar{1} & \bar{0} & \bar{0} \\
            \bar{1} & \bar{1} & \bar{1} & \bar{0} & \bar{1} & \bar{0} \\
            \bar{1} & \bar{1} & \bar{0} & \bar{0} & \bar{0} & \bar{0} \\
            \bar{1} & \bar{0} & \bar{1} & \bar{1} & \bar{1} & \bar{1}
        \end{bmatrix}
        &
        \begin{bmatrix}
         2 \\
         \bar{0}\\
         \bar{0}\\
         \bar{0}\\
         \bar{0}
        \end{bmatrix}
        & 2 & \dfrac{1+t^3}{1-t^4}
        \\[14 pt]
        \hline
        &&&&\\
        154 &
        \dfrac{\K[T_{01},\dots,T_{41},T_1]}{\left\langle \substack{f_1\left(T_{11}^2,T_{21}^2\right)+T_{31}^2,\\ f'_1\left(T_{11}^2,T_{21}^2\right)+T_{41}^2,\\
        f''_1\left(T_{11}^2,T_{21}^2\right)+T_{51}^2}\right\rangle}&
        \begin{bmatrix}
            2 & 1 & 1 & 1 & 1 & 1 & 1\\
            \bar{0} & \bar{1} & \bar{1} & \bar{1} & \bar{1} & \bar{0} & \bar{0}\\
            \bar{0} & \bar{1} & \bar{1} & \bar{1} & \bar{0} & \bar{1} & \bar{0}\\
            \bar{0} & \bar{1} & \bar{1} & \bar{0} & \bar{0} & \bar{0} & \bar{0}\\
            \bar{0} & \bar{1} & \bar{0} & \bar{1} & \bar{1} & \bar{1} & \bar{0}
        \end{bmatrix}
        &
        \begin{bmatrix}
         2 \\
         \bar{0}\\
         \bar{0}\\
         \bar{0}\\
         \bar{0}
        \end{bmatrix}
        & 32 & \dfrac{1+t^3}{1-t^4}
            \end{longtable}
    \end{center}
    }
\end{classification}

\end{document}